\documentclass[a4paper,12pt]{amsart}
\usepackage{amssymb}
\usepackage[utf8]{inputenc}
\usepackage{amsmath}
\usepackage{stmaryrd}
\usepackage{amscd,amsthm,amssymb}
\usepackage{enumerate}
\usepackage{color}
\usepackage[all,cmtip]{xy}
\usepackage{comment}

\scrollmode
\usepackage{latexsym}

\def\.{\cdot}
\def\a{\alpha}
\def\b{\beta}

\def\d{{\mathrm d}}
\def\e{\varepsilon}

\def\la{\langle}
\def\ra{\rangle}

\def\n{\nabla}
\def\nt{\nabla^\tau}

\def\beq{\begin{equation}}
\def\eeq{\end{equation}}
\def\bea{\begin{eqnarray*}}
\def\eea{\end{eqnarray*}}
\def\beaa{\begin{eqnarray}}
\def\eeaa{\end{eqnarray}}
\def\ba{\begin{array}}
\def\ea{\end{array}}
\def\f{\varphi}

\def\o{\omega}
\def\e{\varepsilon}

\def \R{\mathbb{R}}

\def \ZM{\mathbb{Z}}

\def\Re{\mathrm{Re}}
\def\Im{\mathrm{Im}}

\def\Ric{\mathrm{Ric}}

\def\be{\begin{equation}}
\def\ee{\end{equation}}
\def\tr{\mathrm{tr}}

\def\Hol{\mathrm{Hol}}

\def\hol{\mathfrak{hol}}
\def\so{\mathfrak{so}}
\def\su{\mathfrak{su}}

\def\gg{\mathfrak{g}}

\def\u{\mathfrak{u}}

\def\SU{\mathrm{SU}}
\def\U{\mathrm{U}}

\def\D{\mathcal{D}}

\def\G{\mathrm{G}}
\def\H{\mathcal{H}}
\def\L{\mathcal{L}}

\def\SO{\mathrm{SO}}

\def\End{\mathrm{End}}

\def\vol{\mathrm{vol}}

\def\scal{\mathrm{scal}}
\def\Id{\mathrm{id}}

\def\T{\mathrm{\,T}}

\def\V{\mathcal{V}}

\def\Par{\mathrm{Par}}

\newtheorem{epr}{Proposition}[section]
\newtheorem{ath}[epr]{Theorem}
\newtheorem{elem}[epr]{Lemma}

\theoremstyle{definition}
\newtheorem{ede}[epr]{Definition}
\newtheorem{ere}[epr]{Remark}

\title{${\mathrm G}_2$-structures with parallel skew-symmetric torsion}

\author{Andrei Moroianu, Uwe Semmelmann}

\address{Andrei Moroianu \\ Université Paris-Saclay, CNRS,  Laboratoire de mathématiques d'Orsay, 91405 Orsay, France, and Institute of Mathematics ``Simion Stoilow'' of the Romanian Academy, 21 Calea Grivitei, 010702 Bucharest, Romania }
\email{andrei.moroianu@math.cnrs.fr}

\address{Uwe Semmelmann, Institut f\"ur Geometrie und Topologie, Fachbereich Mathematik, Universit{\"a}t Stuttgart, Pfaffenwaldring 57, 70569 Stuttgart, Germany
}
\email{uwe.semmelmann@mathematik.uni-stuttgart.de}

\date{\today}

\begin{document}

\begin{abstract} 
We classify $7$-dimensional Riemannian manifolds carrying a metric connection with parallel skew-symmetric torsion whose holonomy is contained in $\mathrm{G}_2$, up to naturally reductive homogeneous spaces and nearly parallel $\mathrm{G}_2$-structures. This extends and completes the classification initiated by Th. Friedrich in the cocalibrated case. Incidentally, we also obtain the list of $\mathrm{SU}(3)$ geometries with parallel skew-symmetric torsion, up to naturally reductive homogeneous spaces and nearly Kähler manifolds.
\end{abstract}


\maketitle

\section{Introduction}

The Levi-Civita connection, as torsion-free metric connection, is the central object in Riemannian geometry. The next most natural class of connections to study is metric connections with totally skew-symmetric and parallel torsion. These connections have many nice properties, e.g. their curvature is still pair-symmetric, the second Bianchi identity
holds and the first Bianchi identity holds with an additional torsion term. Moreover, these connections have the same geodesics
as the Levi-Civita connection. An important motivation for studying
connections with skew-symmetric parallel torsion comes from the fact that 
they often arise in the presence of special geometric structures, such as nearly K\"ahler and Sasakian manifolds, $\G_2$-structures, and naturally reductive homogeneous spaces. In 
all these cases a canonical connection with parallel skew-symmetric torsion  preserving the structure exists. This connection is often better adapted 
to the special geometry and has interesting additional properties. In recent years, connections with parallel skew-symmetric torsion have attracted substantial interest in mathematics, and various papers have studied them in particular settings (e.g.  \cite{AD20}, \cite{FF25}, \cite{F07}, \cite{PR23}).

Connections with parallel skew-symmetric torsion are also important in theoretical
physics, in particular in superstring and supergravity theories. In type II superstring theory the fundamental string equations were formulated by Strominger among other things in terms of a 3-form. The assumption that this 3-form is the torsion form of a metric connection has proved  to be a successful starting point for a complete solution of the system of string equations (cf. \cite{S86}, see also \cite{FI02}). A particularly interesting dimension in string theory is
dimension $7$. Here the $\G_2$-Hull-Strominger system (or heterotic $\G_2$-system)
was extensively investigated recently (see e.g.   \cite{CGT22} or \cite{OLMS20}).

In their seminal paper \cite{CS04}, Cleyton and Swann presented a 
first classification result for Riemannian manifolds admitting a metric connection with parallel skew-symmetric torsion. Under the assumption that the torsion does not vanish and its holonomy acts irreducibly they show that the manifold is either naturally reductive locally homogeneous, or a nearly K\"ahler manifold
in dimension $6$, or a nearly parallel $\G_2$-manifold in dimension $7$.

In \cite{CMS21} we started a systematic study of geometries with parallel skew-symmetric torsion in the general case, i.e. with reducible holonomy. We showed that in this situation there exists a locally defined Riemannian submersion with totally geodesic leaves, which are naturally reductive spaces. Moreover, the base of the submersion again admits a connection with parallel skew-symmetric torsion, together with some further, somewhat
mysterious, geometric structures. We used this reduction procedure to obtain new classification results under additional assumptions. However, it turned out that there is a huge variety of possible constructions and a complete classification seems to be out of reach for the moment.

The main tool in \cite{CMS21} is the  {\it standard decomposition}
of the tangent bundle into vertical and horizontal parts. Its definition is based on properties of the holonomy algebra of the connection with parallel skew-symmetric torsion (see Sect. \ref{CMS} below). This approach was generalized in \cite{MS24}. It turns out that one can define a standard decomposition, with similar properties, associated to any Lie algebra containing the holonomy algebra and contained in the Lie algebra of the stabilizer of
the torsion form (of course the stabilizer algebra contains the
holonomy algebra). One of the main achievements in \cite{MS24} is the generalization of the classification result in \cite{CS04} to the case where any of these intermediate algebras acts irreducibly on the tangent space.

In the present article, we will change the point of view and study connections with parallel skew-symmetric torsion whose holonomy group is not necessarily irreducible, but is contained in $\G_2$. Note that in terms of the celebrated Fern\'andez-Gray classification \cite{FG82}, a $\G_2$-structure is preserved by a metric connection with skew-symmetric torsion if and only if it belongs to the class $\mathcal{W}_1\oplus\mathcal{W}_3\oplus\mathcal{W}_4$, which in modern language \cite{B06} is characterized by the vanishing of the 2-form part $\tau_2$ of its intrinsic torsion \cite[Thm. 4.7]{FI02}. Note that this class is complementary to the class $\mathcal{W}_2$ of {\em closed} $\G_2$-structures.

In this situation, assuming that the torsion of this connection is parallel, a complete local classification is possible, up to naturally reductive spaces, Joyce manifolds (i.e. Riemannian manifolds with holonomy equal to $\G_2$) and nearly parallel
$\G_2$-manifolds. For these three classes of manifolds many examples are known and at least for Joyce manifolds and nearly parallel $\G_2$-manifolds a classification seems  not feasible. Naturally reductive $7$-dimensional homogeneous spaces have been recently classified by Storm \cite{S20}. However, it is not clear in which cases the holonomy algebra of the canonical homogeneous connection is contained in $\mathfrak{g}_2$.
As a byproduct of our classification in the $\G_2$-case, we also
obtain a classification for  connections with parallel skew-symmetric torsion and holonomy contained in $\SU(3)$.

A related classification result was obtained by Friedrich in \cite{F07}. He considered the special case of cocalibrated $\G_2$-manifolds, i.e. with coclosed $\G_2$-form, assuming non-abelian holonomy and also in most cases assuming the manifolds to be complete and simply-connected. 

With the new tools developed in \cite{CMS21} and \cite{MS24} at hand, we were now able to extend and complete Friedrich's classification. Our main results can be summarized as follows.

\begin{ath}\label{main}
    Let $(M^7, g, \tau, \varphi)$ be a Riemannian manifold with a $\G_2$-structure defined by the $3$-form $\varphi$ and a metric connection $\nabla^\tau$ with skew-symmetric and parallel torsion $\tau$, preserving the $\G_2$-structure, i.e. $\nabla^\tau \tau = 0$ and $\nabla^\tau \varphi = 0$. Let $R^\tau$ be the curvature of $\nabla^\tau$ and let $d: = \dim \Par(\nabla^\tau)$ denote the dimension of the space of $\nabla^\tau$-parallel vector fields. Then $(M,g,\tau,\f)$ is
    locally isomorphic to a manifold in one of the following cases:
    \begin{enumerate}
        \item $\nabla^\tau R^\tau =0$ (in which case $(M,g)$ is a naturally reductive homogeneous space);
        \item $(M,g,\f)$ is a torsion-free $\G_2$-manifold, (i.e. $\Hol(M,\nabla^g) \subset \G_2$), and $\tau = 0$;
        \item $d=1$:
        \begin{enumerate}
        \item $(M,g) = \R \times (N,g^N)$, where $N$ is a $6$-dimensional Calabi-Yau manifold, $\tau=0$ and $\f$ is given by \eqref{fcy};
            \item $(M,g) = \R \times (N,g^N)$, where $N$ is a $6$-dimensional strict nearly K\"ahler manifold and
         $\tau$ and $\varphi$ are given in \eqref{taun}--\eqref{phi};
            \item  $(M,g)$, where $S$ is a $7$-dimensional $\alpha$-Sasaki manifold which is the total space of an $S^1$-fibration over a K\"ahler-Einstein manifold of positive scalar curvature $72\alpha^2$ and $\tau$ and $\f$ are defined in \eqref{tau6}--\eqref{phi6};
            \item $(M,g)$ is the total space of an $S^1$-fibration over the twistor space of an anti-self-dual Einstein manifold in dimension $4$ of positive scalar curvature, and $\tau$ and $\varphi$ are given in \eqref{tf};
            \item $(M,g)$ is the total space of an $S^1$-fibration over a product $(K,g^K)\times (\Sigma,g^\Sigma)$ of Kähler-Einstein manifolds, $K$ of dimension $4$, and $\Sigma$ of dimension $2$, with scalar curvatures satisfying either $\scal^K+\scal^\Sigma>0$ and $\scal^\Sigma\ne 0$, or $\scal^K=\scal^\Sigma=0$ and $\tau$ and $\varphi$ are defined in \eqref{tau7}--\eqref{phi7};
        \end{enumerate}
        \item $d\ge 2 $:
        \begin{enumerate}
        \item $(M,g) = S^3 \times (K,g^K)$ or  $(M,g) = \R^3 \times (K,g^K)$, where $K$ is a $4$-dimensional hyperk\"ahler manifold and $\tau$ and $\varphi$ are given in \eqref{phi-tau3};
        \item  $(M,g) = \R^2 \times (S,g^S) $, where $S$ is a $5$-dimensional $\alpha$-Sasaki manifold which is the total space of an $S^1$-fibration over a K\"ahler-Einstein manifold of positive scalar curvature $32\alpha^2$ and $\tau$ and $\f$ are defined in \eqref{tau5} and \eqref{phi5}.
            \item $(M,g)$ has a parallel $3$-$(\alpha, \delta)$-Sasaki structure with $\delta= 2\alpha$, and $\tau$ and $\varphi$ are given in \eqref{tau4} and \eqref{phi4}.
        \end{enumerate}
        \item $d= 0$:
        \begin{enumerate}
            \item $(M^g,g)$ is a $3$-$(\alpha, \delta)$-Sasaki manifold with $\delta \neq 2\alpha$ and  $\tau$ and $\varphi$ are given in Proposition \ref{ad-thm};
            \item $(M^g,g)$ has a nearly parallel $\G_2$-structure $\f$ and $\tau = \lambda \varphi$ for some $\lambda\in\R^*$.
        \end{enumerate}
    \end{enumerate}   
\end{ath}

\begin{ere}
Note that cases (5)(a) and (5)(b) have an overlap.
Indeed, if $\delta = 5\alpha$, then the $3$-$(\alpha, \delta)$-Sasaki structure defines a nearly parallel $\G_2$-structure by \eqref{tor-ad} with $\tau = \alpha \varphi$.
Similarly, case (2), which corresponds to 
torsion-free $\G_2$-manifolds, i.e. $\tau=0$, has an overlap with (3)(a) and with the second case of (4)(a).

It is easy to compute the holonomy algebra $\hol$ of $\nabla^\tau$ in each of the cases (3)--(5), assuming that $(M,g)$ is not naturally reductive, i.e. not contained in case (1). In cases (3)(a)--(3)(b) one has $\hol=\su(3)\subset\gg_2$, and in cases (3)(c)--(3)(e) one has $\hol=\mathfrak{s}(\u(1)\oplus\u(2))\subset\su(3)\subset\gg_2$. In case (4) one has $\hol = \su(2)\subset\su(3)\subset\gg_2$, in case (5)(a) one has $\hol = \su(2) \oplus \su_c(2)\subset\gg_2$, whereas in case (5)(c), $\hol=\gg_2$.
\end{ere}

The proof of Theorem \ref{main} will be given in Sections \ref{s4}--\ref{xxx}. Note that if a 6-dimensional manifold $(N,g)$ has a metric connection with parallel skew-symmetric torsion whose holonomy is contained in $\SU(3)$, then the induced connection on the Riemannian product $\R\times (N,g)$ has holonomy contained in $\mathrm{SU}(3)\subset\G_2$, and going through the possible cases in Theorem \ref{main} we obtain at once:

\begin{ath}\label{main2}
    Let $(N^6, g)$ be a Riemannian manifold with an $\SU(3)$-structure $(\omega, \psi)$ and a metric connection $\nabla^\sigma$ with skew-symmetric and parallel torsion $\sigma$, preserving the $\SU(3)$-structure, i.e. $\nabla^\sigma\sigma=0$, $\nabla^\sigma \omega = 0$ and $\nabla^\sigma \psi = 0$. Then $(N,g)$ is locally isometric to a manifold in one of the following cases:
    \begin{enumerate}
        \item $(N,g)$ is a naturally reductive homogeneous space;
               \item $(N,g,\omega,\psi)$ is a $6$-dimensional Calabi-Yau manifold, and $\sigma=0$;
            \item $(N,g,\omega)$ is a $6$-dimensional strict nearly K\"ahler manifold, $\sigma=-\frac16*d\omega$, and $\psi$ is a linear combination of $\sigma$ and $*\sigma$.
        \item  $(N,g) = \R \times (S,g^S) $, where $(S,g^S,\xi,\Phi)$ is a $5$-dimensional $\alpha$-Sasaki manifold which is the total space of an $S^1$-fibration over a K\"ahler-Einstein manifold of positive scalar curvature $32\alpha^2$, $\omega=dt\wedge\xi+\Phi^\flat$ and $\sigma=dt\wedge \alpha\Phi$.
    \end{enumerate}   
\end{ath}

Note that in case (4), the $3$-form $\psi$ of the $\SU(3)$-structure is harder to describe. From the considerations in \S\ref{line}, it follows that $\psi$ can be written as $\psi=\xi\wedge\beta_1+dt\wedge\beta_2$, where $\beta_1$ and $\beta_2$ are orthogonal $\nabla^\sigma$-parallel sections of $\Lambda^+(\xi^\perp)\subset\Lambda^2\T S$ of square norm 2. 

The paper is structured as follows.
We start in  Section 2 with collecting a few basic definitions and facts used throughout the article. We discuss in particular $\SU(3)$- and $\G_2$-structures on Euclidean vector spaces, as well as geometric structures on manifolds, such as
Sasaki and nearly K\"ahler structures. At the end of the Section 2 we recall the definition of the standard decomposition and of the corresponding standard submersion for manifolds with a metric connection with parallel skew-symmetric torsion introduced in \cite{CMS21}.

In Section 3 we study the locally defined Riemannian submersion induced by a
$\nabla^\tau$-parallel vector field, where $\nabla^\tau$ is a connection with
skew-symmetric torsion. In particular, we give  in Lemma \ref{ntg} conditions for a $\nabla^\tau$-parallel form on the total space to be projectable to a form on the base which is parallel with respect to an induced connection with skew-symmetric torsion.
In \eqref{rpb} we derive a useful relation between the curvature tensors of these connections. 

In Section 4 we start with the main topic of our article, the classification 
of $7$-dimensional Riemannian manifolds admitting a connection $\nabla^\tau$ 
with parallel skew-symmetric torsion whose holonomy is contained in $\G_2$.  
We divide our study in three cases according to the dimension  of the space
of $\nabla^\tau$-parallel vector fields, denoted by $\Par(\nabla^\tau)$. In the case  $\dim (\Par(\nabla^\tau)) = 1$, the starting point
is the observation in Lemma \ref{hrep} that in this situation the horizontal part $\H$ can have $\hol$-irreducible summands only of dimensions $6$ or $4$. The
corresponding subcases are treated separately and lead to the first classes of examples. 

In Section 5 we study the case $\dim \Par(\nabla^\tau) \ge 2$. In this case one has an orthonormal frame of three $\nabla^\tau$-parallel vector fields spanning a calibrated $3$-plane $\V$, which can be used, together with the torsion form, to define three self-dual forms on the horizontal space $\H:=\V^\perp$. It turns out that these forms (denoted by $\gamma_i$) define a Lie subalgebra of $\Lambda^+\H$, and according to the possible dimensions, $0,\ 1,$ and $3$ of this subalgebra, the solutions to our problem correspond to Riemannian products of 3-dimensional space forms with hyperkähler 4-manifolds, to products of $\R^2$ with Sasakian $S^1$-bundles over Kähler-Einstein 4-manifolds, or parallel $3$-$(\delta, \alpha)$-Sasaki manifolds.

In the final Section 6 we discuss the case  $\dim \Par(\nabla^\tau) = 0$.
Here we use the classification of possible holonomy algebras of $\nabla^\tau$ given
by Friedrich in \cite{F07}. There are five Lie algebras contained in $\mathfrak{g}_2$ acting on $\R^7$ without trivial summands. Four of them
directly lead to naturally reductive locally homogeneous spaces, respectively to  nearly
parallel $\G_2$-manifolds. The remaining case $\hol = \su(2) \oplus \su_c(2)$ leads to further interesting examples 
defined on $3$-$(\delta, \alpha)$-Sasaki manifolds, as described in
Proposition \ref{ad-thm}. Our result in this case can be interpreted as the converse to the construction by Agricola and Dileo \cite[Thm. 4.5.1]{AD20} of the characteristic connection on $3$-$(\delta, \alpha)$-Sasaki manifolds in dimension 7. The key argument here was to use a modified connection which turned out to be flat on the vertical distribution, thus allowing us to recover the Killing vector fields defining the $3$-$(\delta, \alpha)$-Sasaki structure.


 {\sc Acknowledgments.} This research was supported by the Oberwolfach Research Fellows program at the Mathematisches Forschungsinstitut Oberwolfach (MFO). We are grateful to MFO for its hospitality and for providing a stimulating research environment. A.M. was partly supported by the PNRR-III-C9-2023-I8 grant CF 149/31.07.2023 {\em Conformal Aspects of Geometry and Dynamics}. We would also like to thank the anonymous referee for his/her attentive reading of the manuscript and helpful remarks.

\bigskip

\section{Preliminaries}

\subsection{Generalities on multilinear algebra}

Let $(E,\la,\cdot,\cdot\ra)$ be an $n$-dimensional Euclidean space. We identify $E$ with $E^*$ and $\Lambda^2 E\simeq\Lambda^2 E^*$ with the space $\End^-(E)$ of skew-symmetric endomorphisms of $E$ by means of the scalar product. For example, if $X,Y\in E$, then $X\wedge Y$ can be seen as a skew-symmetric endomorphisms of $E$ by the formula
$$(X\wedge Y)(Z):=\la X,Z\ra Y-\la Y,Z \ra X\ .$$

The exterior algebra $\Lambda^*E$ has a unique scalar product extending $\la\cdot,\cdot\ra$ and such that the interior and exterior products with vectors $X\lrcorner$ and $X\wedge$ are adjoint to each other. Note that this scalar product does not correspond to the usual extension of the scalar product to the tensor algebra via the above identification. For instance, if $X,Y\in E$ are two orthogonal vectors, then 
$$|X\wedge Y|^2_{\Lambda^2E}=|X|^2|Y|^2=\frac12|X\wedge Y|^2_{\End^-(E)}\ .$$

Every skew-symmetric endomorphism $A\in\End^-(E)$ extends uniquely as a derivation of the tensor bundle, denoted $A_*$, commuting with the isomorphism $E\simeq E^*$ defined by the scalar product. On the exterior algebra $\Lambda^*E$ one has the convenient expression
\begin{equation}
    A_*=\sum_{i=1}^n A(e_i)\wedge e_i\lrcorner\ ,
\end{equation}
where $\{e_i\}$ is any orthonormal basis of $E$. If $A,B\in\End^-E$, then $A_*B=[A,B]$, where $[A,B]:=AB-BA$ denotes the commutator of endomorphisms.

Other useful formulas which will be needed below are
\begin{equation}\label{degree}
    \sum_{i=1}^n e_i\wedge (e_i\lrcorner \alpha)=p\alpha,\qquad  \sum_{i=1}^n e_i\lrcorner  (e_i\wedge  \alpha)=(n-p)\alpha,\qquad\forall \alpha\in\Lambda^pE\ .
\end{equation}

\subsection{$\SU(3)$-structures}

Let $E$ be a $2m$-dimensional real vector space. A $\U(m)$- (or Hermitian) structure on $E$ is a scalar product $ \la \cdot,\cdot\ra$ together with an orthogonal endomorphism $J$ satisfying $J^2=-\Id_E$. Then $J$ is skew-symmetric and the bilinear form $\omega:=\la J\cdot,\cdot\ra$ is skew-symmetric. 

The extension $J_*$ of $J$ to the exterior bundle $\Lambda^*E$ is skew-symmetric as well. The square $(J_*)^2$ is thus symmetric, and for every $p\ge 0$ its spectrum on $\Lambda^pE$ is given by 
$$\{-(p-2l)^2,\ |\ l\in\ZM\cap[0,\frac p2]\}\ .$$ 
The eigenspace of the restriction of $(J_*)^2$ to $\Lambda^pE$ corresponding to the eigenvalue $-(p-2l)^2$ is denoted by $\Lambda^{(p-l,l)+(l,p-l)}E$, or simply by $\Lambda^{(l,l)}E$ if $p=2l$, so for every $p$ we have the orthogonal direct sum decomposition
$$\Lambda^pE=\bigoplus_{l=0}^{ \lfloor \tfrac{p}{2} \rfloor}\Lambda^{(p-l,l)+(l,p-l)}E\ .$$
The skew-symmetric endomorphisms corresponding to $\Lambda^{(1,1)}E$ and $\Lambda^{(2,0)+(0,2)}E$ via the above identifications are exactly those commuting, respectively anti-commuting, with $J$. The Hodge duality is the isomorphism $*:\Lambda^pE\to\Lambda^{m-p}E$ defined by the volume form $\frac1{m!}\omega^m$. The metric adjoint of $\omega\wedge: \Lambda^pE\to \Lambda^{p+2}E$ is denoted $\Lambda:\Lambda^{p+2}E\to \Lambda^pE$. The kernel of its restriction to $\Lambda^{(k,l)+(l,k)}E$ will be denoted by $\Lambda^{(k,l)+(l,k)}_0E$. 

An $\SU(m)$-structure on $E$ is a Hermitian structure as above, together with an element $\psi\in \Lambda^{(m,0)+(0,m)}E$ satisfying $|\psi|^2_{\Lambda^mE}=2^{m-1}$. The terminology comes from the fact that the subgroup of $\mathrm{GL}(E)$ preserving the structure $(\la\cdot,\cdot\ra,J,\omega,\psi)$ is isomorphic to $\SU(m)$.

Of particular interest for us will be the case $m=3$. 
\begin{elem}\label{op}
    For every $\SU(3)$-structure $(\la\cdot,\cdot\ra,J,\omega,\psi)$ on $E$ there exists an oriented orthonormal basis $\{e_i\}$ of $E$ such that 
    $$\o=e^{12}+e^{34}+e^{56},\qquad\psi=e^{135}-e^{146}-e^{236}-e^{245}\ ,$$
    where $\{e^i\}$ denotes the dual basis and we use the standard notation $e^{ij}:=e^i\wedge e^j$ and $e^{ijk}:=e^i\wedge e^j\wedge e^k$.
\end{elem}
\begin{proof}
    It is clear that one can find an orthonormal basis $\{f_i\}$ of $E$ such that $\omega=f^{12}+f^{34}+f^{56}$. This basis is automatically oriented with respect to the orientation defined by $\omega^3$. Since $\Lambda^{(3,0)+(0,3)}E$ is spanned by the real and imaginary  parts of 
    $$\Psi:=(f^1+if^2)\wedge(f^3+if^4)\wedge(f^5+if^6)\ ,$$  
    there exist $x,y\in\R$ such that $\psi=x\Re(\Psi)+y\Im(\Psi)=\Re[(x-iy)\Psi]$. As $|\psi|^2=|\Re(\Psi)|^2=|\Im(\Psi)|^2$, we have $x^2+y^2=1$. If $z=e^{i\theta}$ denotes a complex number such that $z^3=x-iy$, then $\psi=\Re[z(f^1+if^2)\wedge z(f^3+if^4)\wedge z(f^5+if^6)]$. Thus $\o$ and $\psi$ have the desired form with respect to the basis obtained from $\{f_i\}$ by a rotation of angle $\theta$ in the 2-planes generated by $f_{2j-1},f_{2j}$, for $j=1,2,3$.    
\end{proof}

We list below some standard formulas and facts intensively used all over the paper.

\begin{elem}\label{gen}
    If $(\la\cdot,\cdot\ra,J,\omega,\psi)$ is an $\SU(3)$-structure on a $6$-dimensional real vector space $E$, then the following hold:
    \begin{enumerate}
    \item  $ J_*\sigma=3(*\sigma),\  J_*(*\sigma)=-3\sigma$, and $JX\lrcorner \sigma=-X\lrcorner(*\sigma),\ X\lrcorner \sigma=JX\lrcorner(*\sigma)$ for every $X\in E$ and $\sigma\in\Lambda^{(3,0)+(0,3)}E$;
    \item $E$ is irreducible as $\SU(3)$-representation, and the decompositions of $\Lambda^2E$ and $\Lambda^3E$ in irreducible $\SU(3)$-summands are $\Lambda^2E=\Lambda^{(2,0)+(0,2)}E\oplus \Lambda^{(1,1)}_0E\oplus \R\omega$ and $\Lambda^3E=\R\psi\oplus\R (*\psi)\oplus \Lambda^{(2,1)+(1,2)}_0E\oplus (\omega\wedge E)$;
        \item The map $E\ni X\mapsto X\lrcorner\psi\in \Lambda^{(2,0)+(0,2)}E$ is an isomorphism of $\SU(3)$-representations;
        \item If $\alpha\in\Lambda^{(1,1)}_0E$ then $\alpha_*\psi=\alpha_*(*\psi)=0$.
    \end{enumerate}
\end{elem}

\subsection{$\G_2$-structures} The group $\G_2$ can be defined as the stabilizer in $\mathrm{GL}(7)$ of the $3$-form 
\be\label{stf}\f:=e^{123}+e^{145}+e^{167}+e^{246}-e^{257}-e^{347}-e^{356}\ ,\ee
and is automatically contained in $\SO(7)$ (see  \cite{B87}). A basis in which $\f$ has the above form is called adapted. By the previous observation, all adapted bases induce the same metric and orientation on $\R^7$.

A $\G_2$-structure on a 7-dimensional vector space $F$ is a $3$-form whose stabilizer is isomorphic to $\G_2$, or equivalently, which can be written in the form \eqref{stf} with respect to some basis of $F^*$. Again, such bases will be called adapted, and they all induce the same scalar product, sometimes denoted $g_\f$, and  the same orientation on $F$.

The following result is classical (see e.g. \cite[Prop. 2.3]{FKMS97}):
\begin{elem}\label{tran}
    The group $\G_2$ acts transitively on the set of pairs of orthonormal vectors in $F$.
\end{elem}

We will now describe the relationship between $\G_2$- and $\SU(3)$-structures.

\begin{elem}\label{g2su3}
  Let $\f$ be a $\G_2$-structure on a $7$-dimensional vector space $F$ and let $\xi$ be a unit vector in $F$ (with respect to the induced metric $g_\f$). We denote $E:=\xi^\perp$, $g_0$ the restriction of $g_\f$ to $E$, and decompose $\f=\xi\wedge\o+\psi$, with $\o\in\Lambda^2E$, $\psi\in\Lambda^3E$. Then $(g_0,\o,\psi)$ is an $\SU(3)$-structure on $E$. Conversely, any $\SU(3)$-structure $(g_0,\o,\psi)$ on a $6$-dimensional vector space $E$ induces a $\G_2$-structure $\f:=\xi\wedge\o+\psi$ on $F:=\R\xi\oplus E$, compatible with the metric $g$ on $F$ extending $g_0$ and such that $\xi$ has unit length and is orthogonal to $E$.
\end{elem}
  \begin{proof}
      Using the transitivity of $\G_2$ on the unit sphere of $F$, one can find an adapted orthonormal basis $\{e_i\}$ such that $\xi=e_1$. Then $\{e_2,\ldots,e_7\}$ is an orthonormal basis of $(E,g_0)$, and $(g_0,\o:=e^{23}+e^{45}+e^{67},\psi:=e^{246}-e^{257}-e^{347}-e^{356})$ is an $\SU(3)$-structure. Indeed, $|\psi|^2_{\Lambda^3E}=4$ and an easy calculation gives $\o_*\o_*\psi=-9\psi$.

      Conversely, if $(g_0,\o,\psi)$ is an $\SU(3)$-structure on $E$, Lemma \ref{op} shows (by shifting all indices by $1$), that there exists an orthonormal basis $\{e_2,\ldots,e_7\}$ of $E$ such that $\o=e^{23}+e^{45}+e^{67}$ and $\psi=e^{246}-e^{257}-e^{347}-e^{356}$. Thus $\f:=\xi\wedge\o+\psi$ has the standard form \eqref{stf} by taking $e_1:=\xi$, so in particular is compatible with the metric $g$ defined above.
  \end{proof}

Finally, we describe the $4$-dimensional reduction of spaces with $\G_2$-structures, relative to the choice of a calibrated $3$-plane. Recall that a $3$-plane $P\subset F$ is called calibrated by $\f$ if the restriction of $\f$ to $P$ is a volume form of unit length with respect to the induced metric. Equivalently, $P$ is calibrated if there exists a basis $\{e_i\}$ adapted to $\f$ such that $P$ is spanned by the vectors $e_1, e_2, e_3$.

Let $\f$ be a $\G_2$-structure on $F$ and let $P$ be a calibrated $3$-plane. By definition, there exists an adapted basis $\{e_i\}$ such that $P$ is spanned by $e_1,e_2,e_3$. Then \eqref{stf} reads
\be\label{fb}\f=e^{123}+\sum_{i=1}^3e^i\wedge\beta_i\ ,\ee
where $\beta_1:=e^{45}+e^{67}$, $\beta_2:=e^{46}-e^{57}$, and $\beta^3:=-e^{47}-e^{56}$. We denote by $H$ the orthogonal complement of $P$, i.e. the $4$-dimensional vector space spanned by $e_4,\ldots,e_7$, with the induced metric and orientation. Then for $i\in\{1,2,3\}$ the $2$-forms $\b_i\in\Lambda^2H$ are self-dual, $|\beta_i|^2_{\Lambda^2 H}=2$, and $[\beta_i,\beta_j]=-2\beta_k$ for every even permutation $(i,j,k)$ of $\{1,2,3\}$. Conversely, we have the following:

\begin{elem}\label{phicom} Let $\la\cdot,\cdot\ra_H$ be a scalar product on an oriented $4$-dimensional vector space $H$, and let $\varphi_1,\varphi_2,\varphi_3\in\Lambda^+H$ be self-dual $2$-forms, not all zero, satisfying 
\be\label{beta} [\varphi_i,\varphi_j]=-2\varphi_k,\ \quad\mbox{ for every even permutation } (i,j,k)\mbox{ of } \{1,2,3 \}\ .\ee 
Then $|\varphi_i|^2_{\Lambda^2 H}=2$ for $i\in\{1,2,3\}$. Moreover, if $\{e_1,e_2,e_3\}$ denotes the standard basis of $\R^3$, then the $3$-form $\f$ defined by \eqref{fb} on the $7$-dimensional vector space $F:=\R^3\oplus H$ is a $\G_2$-structure compatible with the direct sum metric and orientation.
\end{elem}
\begin{proof}
        We first notice that for every $i\ne j\in\{1,2,3\}$, if $k$ denotes the index such that $(i,j,k)$ is an even permutation, then \eqref{beta} gives $\la\varphi_i,\varphi_j\ra_{\Lambda^2 H}=-\frac12\la [\varphi_j,\varphi_k],\varphi_j\ra_{\Lambda^2 H}=0$. This shows that $\varphi_1,\varphi_2,\varphi_3$ are mutually orthogonal. If one of $\varphi_i$ vanishes, we get immediately from \eqref{beta} that the two other vanish as well. Thus $\{\varphi_1,\varphi_2,\varphi_3\}$ is an orthogonal basis of $\Lambda^+H$. It is easy to check that every two orthogonal elements of $\Lambda^+H$ anti-commute (as endomorphisms), and for every $\varphi\in\Lambda^+H$ one has $\varphi\circ\varphi=\frac14\tr(\varphi\circ\varphi)\Id_H$. From \eqref{beta} we then obtain
    \bea4\tr(\varphi_k\circ\varphi_k)&=&\tr([\varphi_i,\varphi_j]\circ [\varphi_i,\varphi_j])=-4\tr(\varphi_i\circ\varphi_j\circ\varphi_j\circ\varphi_i)\\
    &=&-\frac14\tr[\tr(\varphi_j\circ\varphi_j)
        \tr(\varphi_i\circ\varphi_i)\Id_H]
    \,=  -\tr(\varphi_j\circ\varphi_j)\tr(\varphi_i\circ\varphi_i)\ .\eea
    As $|\varphi_i|^2_{\Lambda^2 H}=-\frac12\tr(\varphi_i\circ\varphi_i)$ for every $i$, this implies that $|\varphi_i|^2_{\Lambda^2 H}=\frac12|\varphi_j|^2_{\Lambda^2 H}|\varphi_k|^2_{\Lambda^2 H}$ for every even permutation $(i,j,k)$ of $ \{1,2,3 \}$, so $|\varphi_i|^2_{\Lambda^2 H}=2$ for every $i$. 

    The endomorphisms corresponding to $\varphi_i$ are thus complex structures on $H$ compatible with the orientation. We write $\f_1$ in standard form in some orthonormal basis $\{f_1,\ldots,f_4\}$ of $H$ as $\varphi_1=f^{12}+f^{34}$. Since $\varphi_2\in\Lambda^+H$ is orthogonal to $f^{12}-f^{34}\in\Lambda^-H$, we have 
    $$\la\varphi_2(f_1),f_2\ra_H=\la\varphi_2,f^{12}\ra_{\Lambda^2 H}=\frac12\la\varphi_2,f^{12}+f^{34}\ra_{\Lambda^2 H}=\la\varphi_2,\varphi_1\ra_{\Lambda^2 H}=0\ .$$
    Consequently $\varphi_2(f_1)$ is orthogonal to $f_1$ and $f_2$ and has unit length. Up to a rotation in the plane generated by $f_3,f_4$ (which does not change the expression of $\varphi_1$) we can thus assume that $\varphi_2(f_1)=f_3$. Then $\varphi_2(f_3)=-\f_1$. The same argument shows that $\varphi_2(f_2)$ is orthogonal to $f_1$ and $f_2$, but also to $f_3$ (since $\la\varphi_2(f_2),f_3\ra_H=-\la f_2,\varphi_2(f_3)\ra_H=\la f_2,f_1\ra_H=0$. Thus $\varphi_2(f_2)=\e f_4$ and $\varphi_2 e_4=-\e f_2$, i.e. $\varphi_2=f^{13}+\e f^{24}$. As $\varphi_2\in\Lambda^+H$ we must have $\varphi_2=f^{13}-f^{24}$. Finally, $\varphi_3=-\frac12[\varphi_1,\varphi_2]=-\frac12[f^{12}+f^{34},f^{13}-f^{24}]=-(f^{14}+f^{23})$. 

    We have thus shown that there exists an oriented orthonormal basis $\{e_4,\ldots,e_7\}$ (with $e_i:=f_{i-3}$ for $i\in\{4,5,6,7\}$) of $H$ such that $\varphi_1=e^{45}+e^{67}$, $\varphi_2=e^{46}-e^{57}$, and $\varphi^3=-e^{47}-e^{56}$. Thus the form $\f$ defined in \eqref{fb} has the standard form \eqref{stf} with respect to the orthonormal basis $\{e_1,\ldots,e_7\}$ of $\R^3\oplus H$.
    \end{proof}

All the above considerations will be transposed to Riemannian manifolds in the sequel.

\subsection{Sasaki-type structures}

In this article we will meet at several places special types of contact structures. For the convenience of the reader we will recall their definitions.

\begin{ede}\label{a-sas}
    An {\em $\alpha$-Sasaki manifold} $(M, g, \xi, \Phi)$ is a Riemannian manifold $(M^{2n+1},g)$  together with a unit
    Killing vector field $\xi$ and a skew-symmetric endomorphism
    $\Phi$ satisfying
    $$
   \Phi^2 = -\Id + \xi \otimes \xi, \quad  \nabla^g_X\xi = \alpha \Phi(X), \quad
        \nabla^g_X\Phi^\flat = - \alpha X^\flat \wedge \xi^\flat,\qquad\forall X\in \T M\ ,
    $$
    where the $2$-form $\Phi^\flat$ is defined by 
    $\Phi^\flat(X, Y) = g(\Phi X,Y)$ and $X^\flat, \xi^\flat$
    are the dual $1$-forms.
\end{ede}

For $\alpha=1$ this gives the standard definition of a 
{\it Sasaki manifold}. It is easy to check that an $\alpha$-Sasaki manifold
$(M, g, \xi, \Phi)$ becomes a Sasaki manifold
after scaling the metric as $\tilde g = \alpha^2 g$ and the  vector field
as $\tilde \xi = \frac{\xi}{\alpha}$, while keeping the 
endomorphism $\Phi$ unchanged. Note that for the $2$-forms one has
$
\Phi^{\tilde \flat} = \lambda^2 \Phi^\flat
$,
where $\Phi^{\tilde \flat}$ is defined from $\Phi$ using $\tilde g$.

Similarly, there exists a modification of $3$-Sasaki structures introduced in \cite{AD20} under the name of $3$-$(\alpha, \delta)$-Sasaki structure (see also \cite[Sec. 2.11]{MS24}).

\begin{ede}\label{3-sas}
A {\it $3$-$(\alpha, \delta)$-Sasaki manifold} 
$(M, g, \xi_i, \Phi_i),\  i\in\{1,2,3\},$ is a Riemannian
manifold $(M^{4n+3}, g)$ with three unit Killing vector fields $\xi_i$, together with
three skew-symmetric endomorphisms $\Phi_i$ satisfying
\begin{align*}
\xi_k &= - \Phi_i \xi_j = \Phi_j \xi_i \\
\Phi_k X &= - \Phi_i \Phi_j X + g(\xi_j, X)\xi_i = \Phi_j \Phi_i - g(\xi_i, X)\xi_j \qquad \forall X \in \T M \\
d\xi^\flat_i &=2\alpha \Phi^\flat_i + 2(\alpha -\delta) \xi_j \wedge \xi_k 
        =2\alpha \Phi^H_i - 2\delta \xi_j \wedge \xi_k
\end{align*}
where the $2$-forms  $\Phi^H_i$ are defined by 
$\Phi^H_i: = \Phi^\flat_i + \xi_j \wedge \xi_k$.
\end{ede}
For $\alpha = \delta = 1$ one retrieves the classical definition of
 {\it $3$-Sasaki} structures. If $\alpha\delta>0$, every $3$-$(\alpha, \delta)$-Sasaki manifold can be obtained from a $3$-Sasaki manifold by rescaling the metric with different factors on the horizontal and vertical distributions \cite{AD20}.

\subsection{Nearly K\"ahler and Calabi-Yau manifolds}

\begin{ede}
A {\em strict nearly K\"ahler manifold} is a Riemannian manifold
$(M,g)$ together with an almost complex structure $J$ compatible with the
metric and such that $(\nabla^g_XJ)X = 0$ for all $X \in \T M$ and
$\nabla_XJ \neq 0$ for all $X\neq 0$.
\end{ede}

Then we have the following well-known lemma (see e.g. \cite[Lemma 2.4]{BM01}).

\begin{elem} \label{nks}
Let $(M^6, g, J)$ be a strict nearly K\"ahler manifold. Then the $3$-form
$\sigma := - \frac16 * d \omega $ is a non-zero section $\Lambda^{(3,0)+(0,3)}\T M$,
and the metric connection $\nabla^\sigma := \nabla^g + \sigma$ satisfies
$$
\nabla^\sigma J = 0
\qquad \mbox{and} \qquad
\nabla^\sigma \sigma = 0 \ .
$$
\end{elem}

\begin{ede}\label{defcy}
A {\it Calabi-Yau manifold} is a K\"ahler manifold $(M^{2m},g,J)$ with a non-zero $\nabla^g$-parallel complex volume form 
$\Psi \in \Omega^{(m,0)}(M)$.
\end{ede}

Recall that Calabi-Yau manifolds are automatically Ricci-flat. Conversely a Ricci-flat K\"ahler manifold is locally Calabi-Yau.

\subsection{Metric connections with parallel skew-symmetric torsion}\label{CMS}

Let $(M,g)$ be a Riemannian manifold with Levi-Civita connection $\nabla^g$. We will most of the time identify vectors and 1-forms, or skew-symmetric endomorphisms and 2-forms using the metric $g$. Calculations will be done using a local orthonormal frame $\{e_i\}$.

Let $\tau\in\Omega^3(M)$ be a $3$-form. For every vector field $X\in\Gamma(\T M)$ one denotes by $\tau_X$ the endomorphism of $\T M$ defined by $g(\tau_X(Y),Z)=\tau(X,Y,Z)$ for every $Y,Z\in\Gamma(\T M)$, and by $\nt$ the metric connection defined by 
 $$\nt_X:=\n^g_X+(\tau_X)_*,\qquad\forall X\in \T M\ ,$$ 
 whose torsion satisfies $g(T^{\nt}(X,Y),Z)=2\tau(X,Y,Z)$ for every $X,Y,Z\in\Gamma(\T M)$.

Let $\nabla^\tau = \nabla^g + \tau$ be a connection with parallel skew-symmetric
torsion, i.e. $\nabla^\tau \tau = 0$ and denote with $\hol$
the holonomy algebra of $\nabla^\tau$, acting naturally on $\T M$. In \cite{CMS21}
we introduced the {\it standard decomposition } $\T M = \V \oplus \H $
in {\it vertical} and {\it horizontal} directions. 
The horizontal subspace  $\H\subset \T M$ is defined as the sum of $\hol$-irreducible summands
$\H_\alpha$ such that for each $\H_\alpha$ there is an element in $\hol$, which acts non-trivially on $\H_\alpha$ and trivially on $\H_\alpha^\perp$. The subspace $\V = \H^\perp$ is the
direct sum of irreducible summands where no such element exists. 

\begin{ere}\label{par-vf}
    Every $\nabla^\tau$-parallel vector field $\xi$ is tangent to the vertical distribution $\V$ (since the holonomy algebra acts trivially on $\R\xi$).
\end{ere}

The standard decomposition of $\T M$ is $\nabla^\tau$-parallel. 
In \cite[Lem. 3.7]{CMS21} we showed that $\V$ is the vertical distribution 
of  a locally defined Riemannian submersion $\pi: (M,g) \rightarrow (N, g^N)$ with totally geodesic leaves. We call $ \pi $ the {\it standard submersion} of a manifold with parallel skew-symmetric torsion. The fibres of $\pi$ turn out
to be naturally reductive locally homogeneous spaces (see \cite[Prop. 3.13]{CMS21}). In particular, if $\H=0$ then $M$ itself is  a
 naturally reductive locally homogeneous space (see \cite[Rem. 3.14]{CMS21}).
 
The horizontal part of the torsion is projectable to the base $N$
of the standard submersion (cf. \cite[Lemma 3.10]{CMS21}, see also Lemma \ref{Lie} below), where it defines again a connection with parallel skew-symmetric torsion (cf. \cite[Rem. 3.11]{CMS21}, see also Lemma \ref{par} below).

\section{$\nabla^\tau$-parallel vector fields}

The local de Rham theorem states that when a Riemannian manifold $(M,g)$ carries non-zero vector field which is parallel with respect to the Levi-Civita connection $\n^g$, then the manifold is locally isometric to a Riemannian product $\R\times (N,g^N)$. We will now investigate the more general case where the Levi-Civita connection is replaced with a connection $\nt:=\n^g+\tau$ with skew-symmetric torsion. 

 The results in this section are related to the submersion constructions presented in \cite{CMS21} and \cite{MS24} when the vertical distribution has dimension 1. They explain the emergence of $\alpha$-Sasakian manifolds as building blocks in cases 3(c) and 4(b) of Theorem \ref{main}. Note however that the present setting is more general than the ones in \cite{CMS21} and \cite{MS24} since we do not assume that $\nt\tau=0$ for the moment.

Let $\xi$ be a unit length $\nt$-parallel vector field on $(M,g)$. 
\begin{elem}\label{xik}
    The vector field $\xi$ is Killing and satisfies $d\xi=2\tau_\xi$.
\end{elem}
\begin{proof}
    For every $X\in\Gamma(\T M)$ we have
\be\label{nx}\n^g_X\xi=\nt_X\xi-\tau_X\xi=\tau_\xi X\ ,\ee
whence $\nabla^g \xi=\tau_\xi$ is skew-symmetric, so $\xi$ is Killing. Using this we then get
\begin{equation}\label{dxi}
    d\xi=2\nabla\xi=2\tau_\xi\ .
\end{equation}
\end{proof}

We fix as before a local orthonormal frame $\{e_i\}$.

\begin{elem}\label{dbeta}
  For every $\nt$-parallel form $\beta$ on $M$ we have $d\beta=2\sum_i (e_i\lrcorner\tau)\wedge(e_i\lrcorner\beta)$.
\end{elem}
\begin{proof}
   Simple calculation:
   \bea d\beta&=&\sum_i e_i\wedge\n^g_{e_i}\beta=-\sum _ie_i\wedge(\tau_{e_i})_*\beta=-\sum_{i,j} e_i\wedge\tau_{e_i}e_j\wedge (e_j\lrcorner\beta)\\
   &=&\sum_{i,j} e_i\wedge\tau_{e_j}e_i\wedge (e_j\lrcorner\beta)=2\sum_j\tau_{e_j} \wedge (e_j\lrcorner\beta)\ .
   \eea
\end{proof}

We decompose $\T M=\R \xi\oplus\D$, where $\D:= \xi^\perp$, and write correspondingly 
\be\label{taudec}\tau=\xi\wedge\gamma+\sigma\ ,\ee 
with $\gamma:=\xi\lrcorner\tau\in\Omega^2(\D)$, $\sigma\in\Omega^3(\D)$. The notation $\Omega^p(\D)$ stands here for elements $\beta\in\Omega^p(M)$ such that $\xi\lrcorner\beta=0$. 

The unit Killing vector field $\xi$ determines a local Riemannian submersion with 1-dimen\-sional totally geodesic fibers $\pi:(M,g)\to (N,g^N)$, such that $g=\xi\otimes\xi+\pi^*g^N$. We denote by $\nabla^{g^N}$ the Levi-Civita connection of $g^N$. An exterior form $\beta$ on $M$ is the pull-back of a form on $N$ if and only if it is {\it basic}, i.e. $\xi\lrcorner\beta=0$ and $\L_\xi\beta=0$.

\begin{elem}\label{par}
Assume that the components $\gamma$ and $\sigma$ of the torsion form $\tau$ are basic, and write $\gamma=\pi^*(\gamma^N)$, $\sigma=\pi^*(\sigma^N)$ with $\gamma^N\in\Omega^2(N)$, $\sigma^N\in\Omega^3(N)$. 
Let $\tilde\beta \in \Omega^p(M)$ such that $\gamma_*\tilde\beta = 0 = \xi \lrcorner \beta$ and $\nabla^\tau \tilde\beta = 0$. Then there exist a $p$-form $\beta \in \Omega^p(N)$
with $\tilde\beta = \pi^*\beta$ and $\nabla^{\sigma^N}\beta = 0$, where $\nabla^{\sigma^N}:=\nabla^{g^N}+\sigma^N$. 

Conversely, if $\beta \in \Omega^p(N)$ such that $\gamma^N_*\beta = 0$
and  $\nabla^{\sigma^N}\beta = 0$, then $\tilde\beta := \pi^*\beta$
is $\nabla^\tau$-parallel.
\end{elem}
\begin{proof}
We start with a general formula relating the covariant derivatives $\nabla^\tau$ and $\nabla^{\sigma^N}$ in the
Riemannian submersion $\pi : M \rightarrow N$. Let $\beta$ be a $p$-form on $N$.
For any vector fields  $X, Y_1, \ldots Y_p$  on $N$ we denote by $\tilde X, \tilde Y_1, \ldots \tilde Y_p$ their horizontal lifts  to vector fields on $M$
and by $\tilde\beta: = \pi^*\beta$. We then compute:
\begin{equation}\label{general}
\begin{split}
   ( \nabla^\tau_{\tilde X} \tilde\beta )(\tilde Y_1, \ldots, \tilde Y_p)
   =&
   \tilde X(\tilde\beta(\tilde Y_1, \ldots, \tilde Y_p))  
      - \sum_i\tilde\beta(\tilde Y_1, \ldots,\nabla^\tau_{\tilde X }\tilde Y_i, \ldots ,\tilde Y_p)\\
   =&
      \pi^*(X(\beta( Y_1, \ldots, Y_p))) - \sum_i\tilde\beta(\tilde Y_1, \ldots,\nabla^g_{\tilde X }\tilde Y_i
       +\tau_{\tilde X}\tilde Y_i, \ldots ,\tilde Y_p)\\
    =&
       \pi^*(X(\beta(Y_1, \ldots,  Y_p))) - 
       \sum_i\tilde\beta(\tilde Y_1, \ldots,\widetilde{\nabla^{g^N}_X  Y_i}
       + \widetilde{\sigma^N_{X}Y_i}, \ldots ,\tilde Y_p)\\
      =&
     	\pi^*((\nabla^{\sigma^N} \beta)(Y_1, \ldots, Y_p))\ .
\end{split}        
\end{equation}

Let now $\tilde\beta \in \Omega^p(M)$ be $\nabla^\tau$-parallel with 
$\gamma_*\tilde\beta = 0$ and $\xi\lrcorner\tilde\beta=0$. Using Lemma \ref{dbeta} we find
$$
\L_\xi \tilde\beta = \xi \lrcorner d\tilde\beta 
= 2\sum_j \xi \lrcorner(\tau_{e_j} \wedge \tilde\beta_{e_j} )
= -2 \sum_j \gamma_{e_j} \wedge \tilde\beta_{e_j}
= - 2 \gamma_* \tilde\beta
= 0 \ .
$$
Hence $\tilde\beta $ is basic, so we can write $\tilde\beta = \pi^* \beta$
for a $p$-form $\beta \in \Omega^p(N)$. We also have 
$\gamma_*\beta = \pi^*(\gamma^N_*\beta) =0$. From \eqref{general} and the injectivity of $\pi^*$ we conclude that $\nabla^{\sigma^N} \beta = 0$ and $\gamma^N_*\beta =0$.

Conversely, let $\beta \in \Omega^p(N)$ with $\gamma^N_*\beta = 0$ and  $\nabla^{\sigma^N}\beta = 0$. Then by \eqref{general} the pull-back $\tilde\beta:= \pi^* \beta$ satisfies
$(\nabla^\tau_{U_0} \tilde\beta)(U_1, \ldots, U_p) =0$ whenever $U_0, \ldots, U_p$ are horizontal vectors, i.e. orthogonal to $\xi$. Moreover, for every $\U\in \T M$ we have
$$
\xi \lrcorner \nabla^\tau_U \tilde\beta = \nabla^\tau_U(\xi \lrcorner \tilde\beta ) - (\nabla^\tau_U \xi ) \lrcorner \tilde\beta = 0 \ ,
$$
since $\xi\lrcorner\tilde\beta=0$  and $\xi$ is $\nabla^\tau$-parallel. It remains to show that $\nabla^\tau_\xi \tilde\beta$ vanishes when applied to horizontal lifts $\tilde Y_1, \ldots, \tilde Y_p$.  Since a horizontal lift $\tilde Y $ and the vector field
$\xi$ commute, we first obtain  that
$$
\nabla^\tau_\xi \tilde Y = \nabla^g_\xi\tilde Y + \tau_\xi \tilde Y
=  \nabla^g_{\tilde Y}\xi
  + \gamma_* \tilde Y = - \tau_{\tilde Y}\xi + \gamma_*\tilde Y = 2 \gamma_* \tilde Y \ .
$$
Substituting this into the formula for $\nabla^\tau_\xi \tilde\beta$ we find
\begin{align*}
 (\nabla^\tau_\xi \tilde\beta) (\tilde Y_1, \ldots, \tilde Y_p)   
 &=
 \xi(\tilde\beta(\tilde Y_1, \ldots, \tilde Y_p)) -
 \sum_i\tilde\beta (\tilde Y_1, \ldots,\nabla^\tau_\xi \tilde Y_i, \ldots, \tilde Y_p)\\
 &=
 \xi(\pi^*(\beta(Y_1, \ldots, Y_p))) -
 \sum_i\tilde\beta (\tilde Y_1, \ldots, 2\gamma_* \tilde Y_i, \ldots, \tilde Y_p)\\
 &=
 2 (\gamma_* \beta) (\tilde Y_1, \ldots, \tilde Y_p)   = 0
\end{align*}
Combining the three computations above we obtain $\nabla^\tau \tilde\beta =0$.
\end{proof}

We assume from now on that $\nt\tau=0$, i.e. that $\nt$ has {\em parallel} skew-symmetric torsion. Since  $\tau$ and $\xi$ are $\nt$-parallel, the components $\gamma$ and $\sigma$ of $\tau$ defined in \eqref{taudec} are $\nt$-parallel as well.

\begin{elem}\label{txit}
  The action of $\gamma$ (as skew-symmetric endomorphism) on $\tau$ vanishes: $\gamma_*\tau=0$.
\end{elem}
\begin{proof}
Since  $\gamma$ is $\nt$-parallel, we compute using Lemma \ref{dbeta}:
   $$ 0=d(d\xi)=2d(\tau_\xi)=2d\gamma=4\sum_i ({e_i}\lrcorner\tau)\wedge({e_i}\lrcorner\gamma)=4\sum_i\gamma( e_i)\wedge({e_i}\lrcorner\tau)=4\gamma_*\tau\ .
  $$
\end{proof}

By Lemma \ref{txit}, together with the fact that $\gamma_*\xi=0$ and $\gamma_*\gamma=0$, we get 
\begin{equation}\label{ts}
    0=\gamma_*\tau=\gamma_*(\xi\wedge\gamma+\sigma)=\gamma_*\sigma\ .
\end{equation}

\begin{elem}\label{Lie} The forms $\gamma$ and $\sigma$ are basic.
\end{elem}
\begin{proof}
  The forms $\gamma$ and $\sigma$ are $\nt$-parallel and horizontal, i.e. satisfy $\xi\lrcorner\gamma=0$ and $\xi\lrcorner\sigma=0$. Using Lemma \ref{dbeta} and the Cartan formula we compute 
  $$\L_\xi\gamma=d(\xi\lrcorner\gamma)+\xi\lrcorner d\gamma=2\xi\lrcorner\sum_i (e_i\lrcorner\tau)\wedge(e_i\lrcorner\gamma)=2\sum_i (\xi\lrcorner e_i\lrcorner\tau)\wedge(e_i\lrcorner\xi\lrcorner\tau)=0\ ,$$
  and similarly, using also \eqref{ts} together with the fact that $\tau_\xi=\gamma$, we get:
   $$\L_\xi\sigma=d(\xi\lrcorner\sigma)+\xi\lrcorner d\sigma=2\xi\lrcorner\sum_i (e_i\lrcorner\tau)\wedge(e_i\lrcorner\sigma)=2\sum_i (\xi\lrcorner e_i\lrcorner\tau)\wedge(e_i\lrcorner\sigma)=-2(\tau_\xi)_*\sigma=0\ .$$
\end{proof}
For any vector field $X$ on $N$ we denote as before by $\tilde X$ its horizontal lift to $M$. 
\begin{elem}\label{ntg}
    For every $X,Y\in\Gamma(\T N)$ the following relation holds:
    \begin{equation}
        \nabla^{\tau}_{\tilde X}\tilde Y=\widetilde{\nabla^{\sigma^N}_XY}
    \end{equation}
\end{elem}
\begin{proof}
    The usual formula for Riemannian submersions together with \eqref{nx} give
$$\nabla^{g}_{\tilde X}\tilde Y=\widetilde{\nabla^{g^N}_XY}+g(\nabla^{g}_{\tilde X}\tilde Y,\xi)\xi=\widetilde{\nabla^{g^N}_XY}-g(\tilde Y,\nabla^{g}_{\tilde X}\xi)\xi=\widetilde{\nabla^{g^N}_XY}-g(\tilde Y,\tau_\xi{\tilde X})\xi\ ,$$
whence using that $\tau_{\tilde X}\tilde Y=g(\tau_\xi{\tilde X},\tilde Y)\xi+\sigma_{\tilde X}\tilde Y$
$$\widetilde{\nabla^{\sigma^N}_XY}=\widetilde{\nabla^{g^N}_XY}+\widetilde{\sigma^N_XY}=\nabla^{g}_{\tilde X}\tilde Y+g(\tau_\xi{\tilde X},\tilde Y)\xi+\sigma_{\tilde X}\tilde Y=\nabla^{\tau}_{\tilde X}\tilde Y\ .$$
\end{proof}

Using this, together with 
$$g([\widetilde X,\widetilde Y],\xi)=-d\xi(\widetilde X,\widetilde Y)=-2\gamma(\widetilde X,\widetilde Y)\ ,$$
which follows from Lemma \ref{xik}, and 
$$\nabla^\tau_\xi\tilde Z=\nabla^g_\xi\tilde Z+\tau_\xi\tilde Z=\nabla^g_{\tilde Z}\xi+\tau_\xi\tilde Z=-\tau_{\tilde Z}\xi+\tau_\xi\tilde Z=2\gamma(\tilde Z)\ ,$$
which follows from $[\xi,\tilde Z]=0$ and $\tau_\xi=\gamma$,
we readily compute the relation between the curvature tensors $R^\tau$ of $\nabla^\tau$ and $R^{\sigma^N}$ of $\nabla^{\sigma^N}$:
\begin{equation*}\begin{split}\label{rrgen} R^{\tau}_{\tilde X,\tilde Y}\tilde Z=&\nabla^\tau_{\tilde X}\nabla^\tau_{\tilde Y}{\tilde Z}-\nabla^\tau_{\tilde Y}\nabla^\tau_{\tilde X}{\tilde Z}-\nabla^\tau_{[\tilde X,\tilde Y]}\tilde Z\\
=&\widetilde{\nabla^{\sigma^N}_{X}\nabla^{\sigma^N}_{Y}{Z}}-\widetilde{\nabla^{\sigma^N}_{Y}\nabla^{\sigma^N}_{X}{Z}}-\nabla^\tau_{\widetilde{[X,Y]}-2\gamma(\widetilde X,\widetilde Y)\xi}\tilde Z\\
=&\widetilde{R^{\sigma^N}_{X,Y}Z}+4\gamma(\widetilde X,\widetilde Y)\gamma(\tilde Z)\ .
\end{split}
\end{equation*}
By considering $R^\tau$ and $R^{\sigma^N}$ as symmetric endomorphisms of $\Lambda^2(\T M)$ and $\Lambda^2(\T N)$ respectively, the above equality reads:
\be\label{rpb}R^{\tau}(\pi^*\beta)=\pi^*(R^{\sigma^N}(\beta))+4\alpha^2 \la\beta,\omega^N\ra_{\Lambda^2\T N}\pi^*\omega^N\ ,\ee
for every $\beta\in\Omega^2(N)$.

\section{$\G_2$-structures with torsion}\label{s4}

As explained in the introduction, our main objective is to classify $7$-dimensional Riemannian manifolds $(M,g)$ admitting a connection $\nabla^\tau$ with parallel skew-symmetric torsion whose holonomy is contained in $\G_2$. 

We will denote by $\hol$ the holonomy algebra of $\nabla^{\tau}$, identified at each point $x\in M$  with a Lie subalgebra of $\End^-(\T_x M)$. The isomorphism class of the representation of $\hol$ on $\T_x M$ does not depend on $x$ so from now on we will not specify $x$ anymore.

 We will divide our study into three different cases, according to the dimension of trivial summand of $\hol$, i.e. the space of $\nabla^\tau$-parallel vector fields, called $\Par (\nabla^\tau)$.

 If $\Par (\nabla^\tau)\ne0$, one can use the $\nabla^\tau$-parallel vector fields in order to obtain a dimensional reduction. This will be done in this section for the case $\dim(\Par (\nabla^\tau))=1$ and in Section \ref{s5} for the case $\dim(\Par (\nabla^\tau))\ge 2$.  The case $\Par (\nabla^\tau)=0$ will be
 treated in Section \ref{xxx}.

Assume for the remaining part of this section that the space of $\nabla^\tau$-parallel
vector fields is one-dimensional, spanned by a unit vector field $\xi$.  According to the orthogonal splitting $\T M = \R \xi \oplus\D$ we can write
\begin{equation}\label{tau-phi}
\tau = \xi \wedge \gamma + \sigma 
\qquad\text{and}\qquad
\varphi = \xi \wedge \omega +  \psi \ ,
\end{equation}
where $\gamma, \omega \in \Omega^2(\D)$ and  $\sigma, \psi \in \Omega^3(\D)$ and
all these forms are $\nabla^\tau$-parallel. By Lemma \ref{g2su3}, 
$(g|_\D,\omega,\psi)$ defines an $\SU(3)$-structure on $\D_x$ at each point $x\in M$. 

Lemma \ref{gen} gives the decomposition of the exterior powers $\Lambda^k \D$
into $\SU(3)$-irreducible summands, e.g. we have the $\nt$-parallel decomposition 
 $\Lambda^2 \D = \Lambda^{(2,0) + (0,2)}\D \oplus \Lambda^{(1,1)}_0\D \oplus \R \omega$, 
 where the first summand is isomorphic to $\D$ by the $\nabla^\tau$-parallel isomorphism $X \mapsto X \lrcorner \psi$.

\begin{elem}\label{g11}
The $2$-form $\gamma \in \Omega^2(\D)$ defined in \eqref{tau-phi} belongs to
$\Omega^{(1,1)}(\D)$.
\end{elem}
\begin{proof}
    According to the above $\SU(3)$-decomposition of $\Lambda^2 \D$, we can  write $\gamma = X \lrcorner \psi + \gamma_0 + c \omega$, for some $\nt$-parallel vector field  $X\in\Gamma(\D)$, $\nt$-parallel 2-form $\gamma_0\in \Omega^{(1,1)}_0(\D)$ and constant $c\in\R$. But $\Par (\nabla^\tau)=\R\xi$, and $X$ is orthogonal to $\xi$, whence $X=0$. This shows that $\gamma\in \Omega^{(1,1)}(\D)$.
    \end{proof}

By Lemma \ref{Lie}, the forms $\gamma$ and $\sigma$ are basic. Moreover, $\omega$ is $\nabla^\tau$-parallel and $\gamma_* \omega=0$ by Lemma \ref{g11}. Consequently, 
Lemma \ref{par} shows that $\omega$ is basic as well.

Consider the standard decomposition $\T M=\V\oplus\H$ described in \S\ref{CMS} (see also \cite[Def. 3.5]{CMS21}) and assume that $(M,g,\tau)$ is not an Ambrose-Singer manifold (in the sense that the curvature tensor $R^\tau$ of $\nt$ is not $\nt$-parallel). Then $\H\ne 0$ by \cite[Prop. 3.13]{CMS21} and $\H\subset \D$ by Remark \ref{par-vf}. 
We denote by $J$ the $\nt$-parallel complex structure on $\D$ defined by $\omega(\cdot,\cdot)=g(J\cdot,\cdot)$.

\begin{elem}\label{hrep}
    The representation of the holonomy algebra $\hol$ on $\H$ is $J$-invariant, irreducible, and $\dim(\H)\ge 4$. 
 \end{elem}
\begin{proof}
    We decompose $\H=\oplus_\alpha \H_\a$ in irreducible $\hol$ representations as in \cite[Def. 3.3]{CMS21}. By definition, for every $\a$, there exists an element $A_\alpha\in \hol$ acting non-trivially on $\H_\a$ and trivially on $\H_\alpha^\perp$. We denote by $\pi_\alpha:\T M\to\H_\alpha^\perp$ the orthogonal projection, which is clearly $\hol$-equivariant. By irreducibility, $(\pi_\alpha\circ J)|_{\H_\alpha}$ is either zero, or injective. In the latter case $(\pi_\alpha\circ J)(\H_\alpha)$ is a $\hol$-invariant subspace of $\H_\alpha^\perp$ on which $A_\alpha$ acts non-trivially, which is a contradiction. Thus 
$\pi_\alpha\circ J=0$, so $J(\H_\a)=\H_\a$.

    Consequently $\H_\a$ is $J$-invariant for every $\a$. Assume that there exists $\a$ with $\dim(\H_\a)=2$. The corresponding element $A_\a\in\hol$ restricted to $\H_\a$ is then equal to a non-zero multiple of $J|_{\H_\a}$. Then $\tr(A_\alpha J)\ne0$, contradicting the fact that $$\hol\subset\su(3)=\{A\in \End^-(\D)\ |\ [A,J]=0,\ \tr(AJ)=0\}\ .$$

    This shows that $\dim(\H^\alpha)\ge 4$ for every $\a$, and since $\dim(\H)\le 7$, the conclusion follows.
\end{proof}

According to Lemma \ref{hrep}, the possible dimensions of $\H$ are 4 and 6. We will treat the two cases separately in the next two subsections.

\subsection{The case $\dim(\H)=6$}\label{s6} Since the restriction of the standard representation of $\su(3)$ on $\R^6$ to any strict Lie subalgebra is reducible, we must have $\hol\simeq\su(3)$. Consequently, $\gamma$ and $\sigma$ are $\su(3)$-invariant. The trivial summand of the $\su(3)$ representation on $\Lambda^2\R^6$ is spanned by the canonical 2-form and the trivial summand of the $\su(3)$ representation on $\Lambda^3\R^6$ is $\Lambda^{(3,0)+(0,3)}\R^6$. We thus conclude that the components of $\tau$ can be written as $\gamma=\alpha\omega$ and $\sigma=a\psi+b(*_\D\psi)$ for some real constants $a,b,\alpha$. By \eqref{ts} we have 
$$0=\tau_\xi\sigma=\gamma_*\sigma=\alpha J_*\sigma\ .$$
As $J$ acts injectively on $\Lambda^{(3,0)+(0,3)}\R^6$, we must have either $\alpha=0$, or $\alpha \ne 0$ and $\sigma=0$. We treat the two cases separately.

\subsubsection{The case $\alpha=0$} In this case we have $\gamma=0$. In particular, this implies that $\n^g_X\xi=\nt_X\xi-\tau_X\xi=\tau_\xi X=\gamma(X)=0$ for every tangent vector $X\in \T M$, i.e. $\xi$ is $\nabla^g$-parallel, whence $(M,g)$ is locally isometric to a Riemannian product $\R\times (N,g^N)$. We will now gather more information about the structure of $(N,g^N)$.

By Lemma \ref{par}, every $\nt$-parallel form on $M$ which vanishes on $\xi$ is basic. 
 In particular, the forms $\sigma, \omega$ and $\psi$ are projectable to exterior forms on $N$, which we will denote by $\sigma^N, \omega^N$ and $\psi^N$. These forms are $\nabla^{\sigma^N}$-parallel by the same lemma.

Applying Lemma \ref{gen} (1) to the $\SU(3)$-structure $(g^N,\omega^N,\psi^N)$ on $\T N$ and using that $\sigma^N=a\psi^N+b(*_N\psi^N)$, we immediately obtain
\begin{equation}\label{js}J^N_*\sigma^N=3(*_N\sigma^N),\qquad J^NX\lrcorner \sigma^N=-X\lrcorner(*_N\sigma^N),\qquad\forall X\in \T N\ .\end{equation}

We then compute for every $X\in \T N$:
\begin{equation}\label{nj}
    \n^{g^N}_XJ_N=-(\sigma^N_X)_*J^N=(J^N)_*(X\lrcorner\sigma^N)=(J^NX)\lrcorner\sigma^N+X\lrcorner((J^N)_*\sigma^N)=2X\lrcorner(*_N\sigma^N)\ .
\end{equation}

In particular we obtain that $(\n^{g^N}_XJ^N)X=0$ for every $X\in \T N$, so $(N,g^N,J^N,\omega^N)$ is nearly Kähler. 

If $\sigma=0$, the structure is Kähler, and $(J^N,\omega^N,\psi^N)$ is a $\nabla^{g^N}$-parallel $\SU(3)$-structure, so $(N,g^N,J^N)$ is Calabi-Yau (cf. Def \ref{defcy}). 
If $\sigma\ne 0$ then $\n^{g^N}_XJ^N\ne0$, so the nearly Kähler structure is strict.

We will now show that conversely, if $(N^6,g^N,J^N,\omega^N)$ is a 6-dimensional Calabi-Yau manifold or a strict nearly Kähler manifold, then $(M,g):=(N\times \R,g^N+dt^2)$ has a connection $\nt$ with parallel skew-symmetric torsion $\tau$ preserving a family of $\G_2$-structures on $M$. 

In the first case, the Calabi-Yau structure of $N$, i.e. the $\nabla^{g^N}$-parallel $\SU(3)$-structure $(J^N,\omega^N,\psi^N)$, defines for every $x,y\in\R$ with $x^2+y^2=1$ a $\nabla^g$-parallel $\G_2$-structure on $M$ 
\be\label{fcy}\f:=dt\wedge\pi^*\omega^N+x\pi^*\psi^N+y\pi^*(*_N\psi^N)\ .\ee

Assume now that $(N^6,g^N,J^N,\omega^N)$ is strict nearly Kähler and consider the $3$-form $\sigma^N:=-\frac16(*_N\d\omega^N)$. By Lemma \ref{nks}, the connection $\nabla^{\sigma^N}:=\nabla^{g^N}+\sigma^N$ has parallel skew-symmetric torsion and satisfies $\nabla^{\sigma^N}\omega^N=0$.
By the converse statement in Lemma \ref{par}, the connection $\nt$ determined on $M$ by 
\be\label{taun}\tau:=\pi^*\sigma^N\ ,\ee 
where $\pi:N\to N\times \R$ denotes the standard projection on the first factor, has parallel skew-symmetric torsion. For every $x,y\in\R$ we define 
\begin{equation}\label{phi}
\f:=dt\wedge \pi^*\omega^N+x\pi^*\sigma^N+y\pi^*(*_N\sigma^N)\ .
\end{equation} 
By Lemma \ref{g2su3}, the 3-form $\f$ defines a $\G_2$-structure compatible with $g$ if and only if $|\psi|^2_{\Lambda^3N}=4$, which by construction is equivalent to $(x^2+y^2)=\frac4{|\sigma^N|^2_{\Lambda^3\T N}}$.

The components $\omega$ and $\psi$ of the $\G_2$-form $\f$ are $\nt$-parallel by the converse statement in Lemma \ref{par}, and $\nt dt=0$ by direct computation, thus showing that $\nt \f=0$. One can easily check that $\delta^g\f=0$ if and only if $y=0$.

\subsubsection{The case $\alpha\ne 0$ and $\sigma= 0$}\label{412} We denote by $\Phi:=\frac1\alpha\gamma$. Since $\tau=\alpha\xi\wedge\Phi$, $\xi$, $\Phi$ are $\nt$-parallel, and $\Phi^2=-\Id_\D$, we get for every $X\in \T M$:
$$\n^g_X\xi=-\tau_X\xi=\tau_\xi X=\alpha\Phi(X)\ ,$$
and 
$$\n^g_X\Phi=-(\tau_X)\Phi=\Phi_*(\tau_X)=\alpha \Phi_*(g(X,\xi)\Phi-\xi\wedge\Phi(X))=-\alpha\xi\wedge\Phi^2(X)=\alpha \xi\wedge\Phi\ .$$

This shows that $(\xi,\Phi)$ defines an $\alpha$-Sasaki structure on $(M,g)$ (cf. Def. \ref{a-sas}).
Consider as before the local Riemannian submersion $\pi:(M,g)\to (N,g^N)$ determined by $\xi$, and the form $\omega^N\in\Omega^2(N)$ such that $\pi^*\omega^N=\omega$. In this case $\sigma^N=0$, so Lemma \ref{par} shows the well known fact that $\omega^N$ is $\nabla^{g^N}$-parallel, i.e. defines a Kähler structure on $(N,g^N)$. 

We claim that $g^N$ is actually Einstein. Indeed, since $\psi$ is $\nt$-parallel, the image of the curvature operator $R^\tau:\Lambda^2\T M\to\Lambda^2\T M$ acts trivially on $\psi$:
\be\label{rtpsi}(R^\tau(\tilde\beta))_*\psi=0,\qquad\forall \tilde\beta\in\Lambda^2\T M\ .
\ee
The curvature relation \eqref{rpb} reads in the present situation (with $\sigma^N=0$):
\be\label{rrr} R^{\tau}(\pi^*\beta)=\pi^*(R^{g^N}(\beta))+4\alpha^2 \la\beta,\omega^N\ra_{\Lambda^2\T N}\pi^*\omega^N\ ,\ee
for every $\beta\in\Omega^2(N)$. Since $R^{g^N}(\Lambda^2\T N)\subset\Lambda^{(1,1)}\T N$, this shows that $R^\tau(\Lambda^2\D)\subset\Lambda^{(1,1)}\D$. Then from \eqref{rtpsi} we obtain that in fact $R^\tau(\Lambda^2\D)\subset\Lambda^{(1,1)}_0\D$. Taking the scalar product with $\omega$ in \eqref{rrr} we thus obtain for every $\beta\in\Lambda^2\T N$
$$0=\la R^{g^N}(\beta),\omega^N\ra_{\Lambda^2\T N}+12\alpha^2 \la\beta,\omega^N\ra_{\Lambda^2\T N},$$
whence $R^{g^N}(\omega^N)=-12\alpha^2\omega^N$.

On the other hand, on every Kähler manifold  $(N,g^N,J^N,\omega^N)$, one has $R^{g^N}(\omega^N)=-\rho^N$, where $\rho^N=\Ric^N(J\cdot,\cdot)$ is the Ricci form. We have thus obtained that $\Ric^N=12\alpha^2g^N$ so $(N,g^N)$ has positive scalar curvature $\scal^N=72\alpha^2$.

Conversely, let $(N,g^N,\omega^N)$ be a 6-dimensional Kähler-Einstein manifold with positive scalar curvature $\scal^N$. We denote by $\a:=\sqrt{\frac{\scal^N}{72}}$ and let $\zeta$ be a $1$-form such that $d\zeta=2\alpha\omega^N$ on some open set $N_0$. Consider the Riemannian metric on $M:=\R\times N_0$ given by 
$$g=(dt+\pi^*\zeta)^2+\pi^*g^N\ ,$$
where $\pi$ is the projection of the second factor. We denote by $\xi$ the metric dual of $dt+\pi^*\zeta$ and by $\omega:=\pi^*\omega^N$. 

Then
$(M, g, \xi, \Phi)$ is a $7$-dimensional $\alpha$-Sasaki manifold. Indeed, $\xi$ is Killing and satisfies $d\xi=d\pi^*\zeta=2\alpha\pi^*\omega^N$, which is equivalent to the second equation in \eqref{a-sas} for $\Phi=\omega$. Equivalently, $\xi$ is parallel with respect to the connection $\nabla^{g}+\tau$, where 
\be\label{tau6}\tau:=\alpha\xi\wedge \omega\ .\ee 
We also have $(\tau_\xi)_*\omega=\alpha\omega_*\omega=0$ so by Lemma \ref{par}, $\omega$ is $\nabla^{\tau}$-parallel, which immediately gives the last equation in \eqref{a-sas}. 

The curvature operator $R^N$ maps $\Lambda^{(1,1)}_0\T N$ to itself since $(N,g^N,J^N)$ is Kähler-Einstein, vanishes on $\Lambda^{(2,0)+(0,2)}\T N$, and maps $\omega^N$ to $-\frac16\scal^N\omega^N=-12\alpha^2\omega^N$. By \eqref{rrr} we then obtain that $R^{\tau}$ takes values in $\Lambda^{(1,1)}_0\D$. The pair symmetry of $R^{\tau}$ then shows that the restriction of $\nabla^{\tau}$ to $\Lambda^{(3,0)+(0,3)}\D$ is flat, so one can find a (locally defined) $\nabla^{\tau}$-parallel section $\psi$ of $\Lambda^{(3,0)+(0,3)}\D$ of square norm $4$. By the converse statement in Lemma \ref{g2su3}, the 3-form 
\be\label{phi6}\f:=\xi\wedge\omega+\psi\ee 
defines a $\G_2$-structure on $M$ compatible with $g$, which is $\nt$-parallel since its defining components are all $\nt$-parallel.

\subsection{The case $\dim(\H)=4$} The standard decomposition of $\T M$ reads in this case $\T M=\H\oplus \V$ with $\dim(\V)=3$. By Remark \ref{par-vf}, $\xi$ is tangent to $\V$ so one can decompose $\V$ as $\V=\R\xi\oplus \V_0$. Since $\xi\lrcorner\omega=0$, and $\omega$ defines an orthogonal complex structure on $\D=\H\oplus\V_0$, one can write $\omega=\omega_0+e_1\wedge e_2$, where $\omega_0:=\omega|_\H$ determines an orthogonal complex structure on $\H$ and $\{e_1,e_2\}$ is any orthonormal basis of $\V_0$. We decompose 
$$\sigma=e_1\wedge e_2\wedge\zeta+e_1\wedge\sigma_1+
e_2\wedge\sigma_2+\sigma_0\ ,$$
with $\zeta\in \H$, $\sigma_1,\sigma_2\in\Lambda^2(\H)$ and $\sigma_0\in\Lambda^3(\H)$.

We denote as before by $\hol$ the Lie algebra of $\Hol(\nabla^{\tau})$. Since $\nabla^{\tau}$ preserves $\xi$, the $\SU(3)$-structure of $\D$, and the standard decomposition of $\T M$, we have that $\hol\subset\su(3)\cap(\so(\H)\oplus\so(\V))\simeq \su(2)\oplus \mathfrak{u}(1)$. Here $\su(2)=\Lambda^{1,1}_0\H=\Lambda^-\H$ is the set of anti-self-dual $2$-forms on $\H$ and $\u(1)$ is spanned by the 2-form $\omega_0-2e_1\wedge e_2$ on $\T N$. 

Let us denote by $\hol^0$ the projection of $\hol$ on $\su(2)$. Then $\hol\subset \u(1)\oplus \hol^0$, and since $\H$ is irreducible, $\hol$ is non-abelian. Consequently $\hol^0$ is not abelian, so $\hol^0=\su(2)$, i.e. $\Lambda^-\H\subset\hol$. 

This shows that $\hol$ acts irreducibly on $\H$, $\Lambda^-\H$ and $\Lambda^3\H$, whereas it acts trivially on $\Lambda^+\H$. We thus obtain that $\sigma_1,\sigma_2\in \Omega^+(\H)$ and $\zeta=0$, $\sigma_0=0$, i.e. 
$$\sigma=e_1\wedge\sigma_1+e_2\wedge\sigma_2,\qquad\mbox{ with }\sigma_1,\sigma_2\in\Omega^+(\H)\ .$$ 
Similar arguments show that 
$$\gamma=ae_1\wedge e_2+\gamma_0,\qquad\mbox{ with }a\in\R \mbox{ and }\gamma_0\in\Omega^+(\H)\ . $$

From \eqref{ts} we have $\gamma_*\sigma=0$, which now reads
$$0=(ae_1\wedge e_2+\gamma_0)_*(e_1\wedge\sigma_1+e_2\wedge\sigma_2)=a(e_2\wedge\sigma_1-e_1\wedge\sigma_2)+e_1\wedge(\gamma_0)_*\sigma_1+e_2\wedge(\gamma_0)_*\sigma_2\ ,$$
thus implying that 
\begin{equation}\label{gs}
    (\gamma_0)_*\sigma_1=a\sigma_2,\qquad(\gamma_0)_*\sigma_2=-a\sigma_1 \ .
\end{equation} 

Recall that $\Lambda^-\H\subset\hol\subset\Lambda^-\H\oplus \u(1)$ where $\u(1)$ is spanned by $\eta:=\omega_0-2e_1\wedge e_2\in\Omega^2(\T N)$. For dimensional reasons we thus either have $\hol=\Lambda^-\H$ or $\hol=\Lambda^-\H\oplus \u(1)$.

However, the case $\hol=\Lambda^-\H$ is impossible, since $\V$ would then be a trivial $\hol$-represen\-tation, which contradicts the assumption $\dim(\Par (\nabla^\tau))=1$ valid throughout this section.

Consequently $\hol=\Lambda^-\H\oplus \u(1)$ so in particular it contains the element $\eta:=\omega^0-2e_1\wedge e_2$, whence 
\begin{equation}\label{ge}
 \eta_*\gamma=0,\qquad   \eta_*\sigma=0\ .
\end{equation}
The first equation is equivalent to $[\omega_0,\gamma_0]=0$, i.e. $\gamma_0$ commutes (as endomorphism) with $\omega_0$. Since they both belong to $\Omega^+(\H)$, and are $\nabla^{\tau}$-parallel, they must be proportional, so there exists $b\in\R$ such that 
\be\label{gab}\gamma=ae_1\wedge e_2+b\omega_0\ .\ee

On the other hand, if we write $\omega=\eta+3e_1\wedge e_2$, and use the fact that $\eta_*$ commutes with $(e_1\wedge e_2)_*$, the second equation in \eqref{ge} shows that $\omega_*\omega_*\sigma=9(e_1\wedge e_2)_*(e_1\wedge e_2)_*\sigma=-9\sigma$. Consequently, $\sigma$ is of type $(3,0)+(0,3)$ with respect to the complex structure defined by $\omega$ on $\D$. 

Writing now $\gamma=\frac{a+2b}3\omega+\frac{b-a}3\eta$ and using the second equation in \eqref{ge} together with \eqref{ts} and Lemma \ref{gen} (1), we obtain
\be\label{a+2b}0=\gamma_*\sigma=\frac{a+2b}3\omega_*\sigma=(a+2b)*\sigma\ .\ee
Therefore we either have $a+2b=0$, or $\sigma=0$. These cases will be treated separately in the next two subsections.

\subsubsection{The case $a+2b=0$} In this case we have $\gamma_*\psi=0$, so $\psi$ is projectable onto a $\nabla^{\sigma^N}$-parallel $(3,0)+(0,3)$-form $\psi^N$ on $N$.
We can then apply the formulas \eqref{js} and \eqref{nj} to deduce as before that $(N,g^N,J^N)$ is strict nearly Kähler. Moreover, its canonical nearly Kähler connection (which is $\nabla^{\sigma^N}$) has holonomy contained in $\mathfrak{s}(\u(1)\oplus\u(2))$ so by \cite[Thm. 5.1]{MS24}, $(N,g^N,J^N)$ is the twistor space of an anti-self-dual 4-dimensional Einstein manifold with positive scalar curvature.

Conversely, assume that $(N,g^N,J^N)$ is the nearly Kähler twistor space of an anti-self-dual 4-dimensional manifold with positive scalar curvature, and let $\sigma^N$ be the $3$-form defined by $\sigma^N:=-\frac16 *_N d\omega^N$. By Lemma \ref{nks}, $\nabla^{\sigma^N}$ is the canonical connection of the nearly Kähler structure, and it is well known (see for instance \cite{R93}) that in this case it preserves both the vertical space $\V_0$ and the horizontal space $\H$ of the twistor fibration. We denote by $\omega_\H$ and $\omega_\V$ the restrictions of $\omega^N$ to $\H$ and $\V_0$. 

\begin{elem}
  The $2$-form $\eta:=\omega_\H-2\omega_\V\in\Omega^2(N)$ is closed.
\end{elem}
\begin{proof}
    Since $\eta$ is $\nabla^{\sigma^N}$-parallel, Lemma \ref{dbeta} gives:
\bea
d\eta&=&2\sum_{i=1}^6(e_i\lrcorner\sigma^N)\wedge(e_i\lrcorner\eta)
\, = \, 2\eta_*\sigma^N\ .
\eea
On the other hand, $\eta\in \Omega^{1,1}_0(N)$ is a primitive $(1,1)$ form, whereas $\sigma^N \in \Omega^{(3,0)+(0,3)}(N)$, and Lemma \ref{gen} (4) shows that the former space acts trivially as derivation on the latter. 
\end{proof}

Consequently, on any contractible open subset $N_0$ of $N$, there exists a 1-form $\zeta\in \Omega^1(N_0)$ such that $d\zeta=\eta$. We define the Riemannian metric 
\begin{equation}\label{g}
    g:=(dt+\pi^*\zeta)^2+\pi^*g^N
\end{equation} 
on the manifold $M:=\R\times N_0$, where $\pi$ denotes the projection on the second factor. Then $\pi$ is a Riemannian submersion with totally geodesic fibers tangent to the unit vector field $\xi:=\partial/\partial t$ (which is the metric dual of $dt+\pi^*\zeta$). 

For all real numbers $x,y$ we define the exterior forms on $M$:
\begin{equation}\label{tf}
  \tau:=\frac12(dt+\pi^*\zeta)\wedge\pi^*\eta+\pi^*\sigma^N,\qquad \f:= (dt+\pi^*\zeta)\wedge\pi^*\omega^N+x\pi^*\sigma^N+y\pi^*(*_N\sigma^N)\ .
\end{equation} 
Like before, Lemma \ref{g2su3} shows that $\f$ defines a $\G_2$-form on $M$ which is compatible with $g$ if and only if $(x^2+y^2)=\frac4{|\sigma^N|^2_{\Lambda^3\T N}}$. 

We claim that $\tau$ and $\f$ are parallel with respect to the metric connection $\nt:=\n^g+\tau$. Using the converse statement in Lemma \ref{par}, we need to show that 
\begin{enumerate}
    \item $\eta,\ \omega^N$, and $\sigma^N$ are $\nabla^{\sigma^N}$-parallel;
    \item $\eta_*\omega^N=0$, and $\eta_*\sigma^N=\eta_*(*_N\sigma^N)=0$;
   \item $\nt(dt+\pi^*\zeta)=0.$
\end{enumerate}
The first item is clear by construction, and the second follows from the fact that $\eta\in\Omega^{(1,1)}(N)$ so its action $\eta_*$ on $\omega^N$ and $\Lambda^{(3,0)+(0,3)}\T N$ is zero by Lemma \ref{gen} (4).
Finally, since the dual vector field $\xi$ of $dt+\pi^*\zeta$ is Killing with respect to $g$, we compute for every $X\in \T M$:
$$\nabla^\tau_X\xi=\nabla^g_X\xi+\tau_X\xi=\nabla^g_X\xi-\tau_\xi X=\frac12 d(dt+\pi^*\zeta)(X)-\frac12\pi^*\eta(X)=0\ .$$

It is easy to check that the $\G_2$-structure defined in \eqref{tf} is cocalibrated if and only if $y=0$. 

 \begin{ere} Note that this solution (corresponding to the case (3)(d) in Theorem \ref{main}) is somewhat similar to the one in \S4.1.2, corresponding to the case (3)(d) in Theorem \ref{main}. Both are obtained as Riemannian submersions with totally geodesic 1-dimensional fibers over special 6-manifolds, except that in the present situation the metric is the nearly Kähler metric of a twistor space over anti-self-dual positive Einstein 4-manifold, whereas in \S4.1.2 the metric is Kähler-Einstein. It is well-known that every such twistor space also admits a Kähler-Einstein metric, which means that every solution in the case (3)(d) gives rise to a solution in the case (3)(c) of Theorem \ref{main} (but not conversely). Moreover, the torsion forms of these solutions are of different algebraic nature.\end{ere}

We now go back to \eqref{a+2b} and consider the other possibility.

\subsubsection{The case $\sigma=0$} \label{422}
This case is somewhat similar to \S\ref{412} but more involved.
Consider as before the local Riemannian submersion $\pi:(M,g)\to (N,g^N)$ determined by $\xi$, and the forms $\omega^N,\ \gamma^N,\ \eta^N\in\Omega^2(N)$ and such that $\pi^*\omega^N=\omega$, $\pi^*\eta^N=\eta$, and $\pi^*\gamma^N=\gamma$. In this case $\sigma^N=0$, so Lemma \ref{par} shows that $\omega^N,\ \gamma^N$, and $\eta^N$ are $\nabla^{g^N}$-parallel. Thus $\omega^N$ defines a Kähler structure on $(N,g^N)$, and $\eta^N$ defines a $\nabla^{g^N}$-parallel splitting $\T N=\H^N\oplus\V_0^N$ whose horizontal lift is exactly the decomposition $\D=\H\oplus\V_0$. 

By the local de Rham theorem, $(N,g^N)$ is locally the product of two Kähler manifolds: $(K,g^K,\omega^K)$ of real dimension $4$, and $(\Sigma,g^\Sigma,\omega^\Sigma)$ of real dimension 2. By \eqref{gab}
we have 
\be\label{gn}\gamma^N=a\omega^K+b\omega^\Sigma\ .\ee
We will now show that, similar to \S\ref{412}, $g^K$ and $g^\Sigma$ are Einstein metrics. 

The curvature relation \eqref{rpb} reads in the present situation (with $\sigma^N=0$):
\be\label{rrr2} R^{\tau}(\pi^*\beta)=\pi^*(R^{g^N}(\beta))+4 \la\beta,a\omega^K+b\omega^\Sigma\ra_{\Lambda^2\T N}\pi^*(a\omega^K+b\omega^\Sigma)\ ,\ee
for every $\beta\in\Omega^2(N)$.

Like before, since $\psi$ is $\nt$-parallel, the image of the curvature operator $R^\tau:\Lambda^2\T M\to\Lambda^2\T M$ acts trivially on $\psi$:
\be\label{rtpsi1}(R^\tau(\tilde\beta))_*\psi=0,\qquad\forall \tilde\beta\in\Lambda^2\T M\ .
\ee
 Since $R^{g^N}(\Lambda^2\T N)\subset\Lambda^{(1,1)}\T N$, \eqref{rrr2} shows that $R^\tau(\Lambda^2\D)\subset\Lambda^{(1,1)}\D$. Then from \eqref{rtpsi1} we obtain that in fact $R^\tau(\Lambda^2\D)\subset\Lambda^{(1,1)}_0\D$. Taking the scalar product with $\omega$ in \eqref{rrr2} we thus obtain for every $\beta\in\Lambda^2\T N$
$$0=\la R^{g^N}(\beta),\omega^N\ra_{\Lambda^2\T N}+4(2a+b) \la\beta,a\omega^K+b\omega^\Sigma\ra_{\Lambda^2\T N},$$
whence $R^{g^N}(\omega^N)=-4(2a+b)(a\omega^K+b\omega^\Sigma)$.

On the other hand, $R^{g^N}(\omega^N)=-\rho^K-\rho^\Sigma$, whence $\Ric^K=4a(2a+b)g^K$ and $\Ric^\Sigma=4b(2a+b)g^\Sigma$. In particular $\Sigma$ has constant Gaussian curvature. The scalar curvatures are then given by 
\be\label{scal}\scal^K=16a(2a+b),\qquad\scal^\Sigma=8b(2a+b)\ .\ee 
If $b=0$ we get $\scal^\Sigma=0$, so $\Sigma$ is flat, and actually a direct factor in $M$ (as $\tau=a\pi^*\omega^K$ by \eqref{gn} in this case). This contradicts the fact that $\Par(\nabla^\tau)=\R\xi$. Consequently $b\ne0$.

We thus either have $a+2b\ne 0$ and $b\ne 0$, in which case $\scal^K+\scal^\Sigma=8(2a+b)^2>0$ and $\scal^\Sigma\ne 0$, or $a+2b=0$ and $b\ne 0$, in which case $\scal^K=\scal^\Sigma=0$.

Conversely, consider a Kähler-Einstein manifold $(K,g^K,\omega^K)$ of real dimension $4$ and a constant curvature surface $(\Sigma,g^\Sigma,\omega^\Sigma)$, with scalar curvatures satisfying either $\scal^K+\scal^\Sigma>0$ and $\scal^\Sigma\ne 0$, or $\scal^K=\scal^\Sigma=0$. Then there exist solutions $a,b$ of the system \eqref{scal} with $b\ne 0$: in the first case the solution is up to sign uniquely determined by
$$a=\frac{\scal^K}{4\sqrt{2(\scal^K+\scal^\Sigma)}},\qquad b=\frac{\scal^\Sigma}{2\sqrt{2(\scal^K+\scal^\Sigma)}}\ ,$$ whereas in the second case there is a 1-parameter family of  solutions of the form $(a,b)=(t,-2t)$, with $t\ne 0$. 

For any such solution $(a,b)$, consider the Riemannian product $$(N,g^N,\omega^N)=(K,g^K,\omega^K)\times (\Sigma,g^\Sigma,\omega^\Sigma)$$ 
and let $\zeta$ be a primitive of $\gamma^N:=a\omega^K+b\omega^\Sigma$ on some open set $N_0$. Consider the Riemannian metric $g:=(dt+\pi^*\zeta)^2+\pi^*g^N\ ,$ on $M:=\R\times N_0$, where $\pi$ is the projection of the second factor. We denote by $\xi$ the metric dual of $dt+\pi^*\zeta$, by $\omega:=\pi^*(\omega^K+\omega^\Sigma)$, by $\gamma:=\pi^*\gamma^N$, and 
\be\label{tau7}\tau:=\xi\wedge \gamma\ .
\ee

Then $\xi$ is a unit Killing vector field on $(M,g)$ and satisfies $d\xi=d\pi^*\zeta=2\pi^*\gamma^N=2\gamma$, whence $\nabla^g_X\xi=\gamma(X)$ for every $X\in\T M$. Equivalently, $\xi$ is parallel with respect to the connection $\nt:=\nabla^{g}+\tau$. We also have $\gamma_*\omega=\pi^*([\gamma^N,\omega^N])=\pi^*([a\omega^K+b\omega^\Sigma,\omega^K+\omega^\Sigma])=0$, so by Lemma \ref{par}, $\omega$ is $\nabla^{\tau}$-parallel. Thus $\nt$ is a metric connection with parallel skew-symmetric torsion.

By doing the above calculations in reverse order, we obtain that $R^{\tau}$ takes values in $\Lambda^{(1,1)}_0\D$, and as before  one can find a (locally defined) $\nabla^{\tau}$-parallel section $\psi$ of $\Lambda^{(3,0)+(0,3)}\D$ of square norm $4$. By the converse statement in Lemma \ref{g2su3}, the 3-form 
\be\label{phi7}\f:=\xi\wedge\omega+\psi\ee 
defines a $\nt$-parallel $\G_2$-structure on $M$ compatible with $g$.

\section{The case $\dim(\Par (\nabla^\tau))\ge 2$}\label{s5}

In this section we assume that there exist (at least) two orthogonal $\nabla^\tau$-parallel unit vector fields $\xi_1$ and $\xi_2$. 
Then $\xi_3 := \varphi(\xi_1, \xi_2)$ is a $\nabla^\tau$-parallel unit vector field, orthogonal to $\xi_1$ and $\xi_2$. Indeed, since $\G_2$ acts transitively on orthonormal pairs of vectors, one can find for every $x\in M$ an adapted orthonormal basis $\{e_1,\ldots,e_7\}$ of $\T_x M$ such that $\xi_1=e_1$ and $\xi_2=e_2$, so by \eqref{stf} one gets that $\xi_3=e_3$ is also a unit vector. Let $\V := \R\xi_1\oplus \R\xi_2\oplus \R\xi_3$
and $\H: = \V^\perp$ and $\T M = \V \oplus \H$ be the corresponding $\nt$-parallel orthogonal splitting of $\T M$. We assume that $(M,g,\tau)$ is not Ambrose-Singer, so this is exactly the standard decomposition of $\T M$ by Lemma  \ref{hrep}. Using the expression \eqref{fb} of the $\G_2$-structure $\varphi$, one can write
\begin{equation}\label{phi1}
\varphi = \xi_1 \wedge \xi_2 \wedge \xi_3 + \sum^3_{i=1} \xi_i \wedge \beta_i \ .
\end{equation}
for some self-dual $2$-forms $\beta_i\in\Omega^+(\H)$ which satisfy $[\beta_i, \beta_j] = - 2 \beta_k$ for
all even permutations $(i,j,k)$ of $\{1,2,3\}$, are pairwise orthogonal, and with $|\beta_i|^2_{\Lambda^2\H}=2$.
Moreover, the $2$-forms $\beta_i$ are $\nabla^\tau$-parallel, thus the holonomy algebra $\hol$ acts
trivially on $\beta_i$, whence $\hol \subset \Lambda^-\H$. The inclusion cannot be strict since then the holonomy algebra would be at most 1-dimensional, contradicting the fact that $\H$ is $\hol$-irreducible by Lemma  \ref{hrep}.

We thus have $\hol = \Lambda^- \H\simeq\su(2)$. According to the splitting $\T M=\R\xi_1\oplus \R\xi_2\oplus \R\xi_3\oplus\H$, we write the torsion form $\tau$ as
$$
\tau = a \xi_1 \wedge \xi_2 \wedge \xi_3 
+ \sum_{i,j} \xi_j \wedge \xi_k \wedge \eta_i 
+ \sum^3_{i=1} \xi_i \wedge \gamma_i 
+ \tau_\H   \ ,
$$
for some constant $a\in\R$, and horizontal forms $\eta_i \in \Omega^1(\H), \gamma_i \in \Omega^2(\H)$
and $\tau_\H \in \Omega^3(\H)$. From the $\hol$-invariance of these forms we
immediately have $\eta_i = 0$, $\tau_\H =0$ and $\gamma_i \in \Omega^+(\H)$.
Hence the torsion form can be written as
\begin{equation}\label{tau1}
\tau = a \xi_1 \wedge \xi_2 \wedge \xi_3  
+ \sum^3_{i=1} \xi_i \wedge \gamma_i  \ .
\end{equation}

\begin{elem}\label{com}
For every even permutations $(i,j,k)$ of $\{1,2,3\}$ it holds that
    $[\gamma_i, \gamma_j] = a \gamma_k$.
    
\end{elem}
\proof
From Lemma \ref{txit} we know that the action of $\tau_{\xi_i}$ on $\tau$
vanishes. Hence:
$$
0 = (\tau_{\xi_i})_*\tau = (a \xi_j \wedge \xi_k + \gamma_i)_*\tau
= a \xi_k \wedge \gamma_j - a \xi_j \wedge \gamma_k + \xi_j \wedge [\gamma_i, \gamma_j]
+ \xi_k \wedge [\gamma_i, \gamma_k]
$$
for all even permutations $(i,j,k)$  of $\{1,2,3\}$ and  the claimed commutator relation follows.
\qed

Consider the $3\times 3$ matrix $A=(a_{ij})$ defined by $\gamma_i = \sum^3_{j=1} a_{ij} \beta_j$.
\begin{elem}\label{det}
For every $i\in\{1,2,3\}$ one has:
   \be\label{aa}a |\gamma_i|^2_{\Lambda^2\H} = -4 \det A\ .\ee 
\end{elem}
\proof
The elements $\beta_i\in\Omega^+(\H)$ satisfy $|\beta_i|^2_{\Lambda^2\H}=2$ and $[\beta_i,\beta_j]=-2\beta_k$ for every even permutation $(i,j,k)$ of $\{1,2,3\}$. This shows that for every $r,s,t\in\{1,2,3\}$ one has that $g([\beta_r, \beta_s], \beta_t)=-4\e(r,s,t)$ where $\e(r,s,t)$ is the signature of the permutation $(r,s,t)$ if the indices are mutually distinct, and $0$ otherwise.
From the definition of the matrix $A$ we then immediately obtain
$$
g([\gamma_1, \gamma_2], \gamma_3)  
= \sum_{r,s,t} a_{1r} a_{2s} a_{3t} \, g([\beta_r, \beta_s], \beta_t)
=-4 \det A \ .
$$
The conclusion follows from Lemma \ref{com}.
\qed

The space generated by $\gamma_1,\gamma_2,\gamma_3$ is a Lie subalgebra of $\su(2)$. We will distinguish three cases, according to the possible dimensions of this Lie algebra.

\subsection{The case $\gamma_1=\gamma_2=\gamma_3=0$}
In this case the torsion
form $\tau $ is a section of $\Lambda^3 \V$, so the geometry
with torsion is decomposable in the sense of \cite[Def. 3.1]{CMS21}. By \cite[Lemma 3.2]{CMS21}, the manifold $(M,g)$ is
locally isometric to a Riemannian product of two Riemannian manifolds $(M^3, g_3)$ and  $(M^4,g_4)$, and by \eqref{tau1}, $\tau$ can be identified with $a\, \vol^{g_3}$. 

The $\nt$-parallel vector fields $\xi_i$ on $M$ satisfy in particular $\n^g_X\xi_i=0$ for every $X\in \H$ so they are constant along $M^4$, and 
\be\label{nxi}\n^g_{\xi_i}\xi_j=-\tau_{\xi_i}\xi_j,\qquad\forall i,j\in\{1,2,3\}\ .\ee

The 2-forms $\beta_i$ defined in \eqref{phi1} satisfy
$\n^g_X\beta_i=0$ for every $X\in\T M$, so they are constant along $M^3$ and define a hyperkähler structure on $(M^4,g_4)$.
By \eqref{nxi}, the vector fields $\xi_i$ satisfy with respect to the Levi-Civita connection on $(M^3,g_3)$ the equations
$$
\nabla_{\xi_i} \xi_i = 0, \qquad \nabla_{\xi_i} \xi_j = a \xi_k, \qquad
 \nabla_{\xi_j} \xi_i = -a \xi_k \ 
$$
for every even permutation $(i,j,k)$ of $\{1,2,3\}$ (we drop the reference to $g_3$ in this paragraph and denote $\nabla^{g_3}$ simply by $\nabla$).
As a first consequence we have $d\xi_i = - 2a \xi_j \wedge \xi_k$. We also
have $\delta \xi_i = 0$, since $\xi_i$ are Killing vector fields. We claim that the manifold $(M^3,g_3)$
has constant sectional curvature. To see this, we first compute
$$
\nabla^* \nabla \xi_i = - \nabla_{\xi_j}\nabla_{\xi_j} \xi_i 
        - \nabla_{\xi_k}\nabla_{\xi_k} \xi_i = a\nabla_{\xi_j}\xi_k 
        - a \nabla_{\xi_k} \xi_j = 2a^2 \xi_i \ .
$$
Since $\xi_i$ are Killing vector fields we obtain
$$
\Ric(\xi_i) = \frac12 \Delta \xi_i = \nabla^* \nabla \xi_i = 2a^2 \xi_i  \ .
$$
Hence, the sectional curvature of $(M^3, g_3)$ is equal to $a^2$, i.e. $M_3$ is locally isometric to the sphere of radius $\frac1{|a|}$ for $a\neq 0$ and to $\R^3$ for $a=0$.

Conversely, let $(M,g) = (M^3,g_3) \times (M^4,g_4)$ the Riemannian product of an oriented manifold
$(M^3, g_3)$ of constant sectional curvature $a^2$ and a hyperk\"ahler manifold $(M^4,g_4)$. For $\tau := a\vol^{g_3}$, a straightforward computation shows that the connection $\nabla^\tau:=\nabla^{g_3}+\tau$ is flat, so there exists an oriented local orthonormal frame of $\nt$-parallel vector fields $\xi_i$ on $M^3$. For any $\nabla^{g_4}$-parallel frame $\{\beta_i\}$ of $\Lambda^+ \T M^4$ satisfying 
the anti-quaternionic relations, the 3-forms on $M$ defined by
\be\label{phi-tau3}\varphi: = \xi_1\wedge\xi_2\wedge\xi_3 +\sum_i \xi_i \wedge \beta_i,\qquad \tau := a\xi_1\wedge\xi_2\wedge\xi_3\ee
determine a $\G_2$-structure $\f$ and a connection $\nt$ with skew-symmetric torsion such that $\nt\tau=0$ and $\nt\f=0$.

\subsection{The case where $\gamma_1,\gamma_2,\gamma_3$ span a real line}\label{line} In this case one can write $\gamma_i = v_i \nu$, for some constants $v_i$ with $v_1^2+v_2^2+v_3^2=1$ and a non-zero $\nabla^\tau$-parallel 2-form $\nu\in \Omega^+(\H)$.
Let $v\in\R^3$ be the unit vector with components $v_i$ and let $B =(b_{ij}) \in \SO(3)$ be such that $Bv = e_1$. We define the unit $\nabla^\tau$-parallel vector fields
$\tilde \xi_i: = \sum_j b_{ij} \xi_j$ and the 2-forms $\tilde \gamma_i: = \sum_j b_{ij} \gamma_j = \sum_j (b_{ij} v_j)\nu$. Since $\det(A)=0$, Lemma \ref{det} gives $a=0$, so $\tau=\sum_j \xi_i \wedge \gamma_i  = \sum_j \tilde \xi_i \wedge \tilde \gamma_i$. Moreover, from
the definition of $B$ we have $\tilde\gamma_1=\nu$ and $\tilde\gamma_2=\tilde\gamma_3=0$, whence $\tau = \tilde\xi_1 \wedge \nu$.

The special form of the torsion $\tau $ in this case implies
$\nabla^g \tilde\xi_2 = \nabla^g \tilde\xi_3 = 0$. 
Hence, one can write locally $(M,g) = \R^2 \times (S,g^S)$ for some $5$-dimensional manifold $(S,g^S)$.
Moreover, since $\tilde\xi_1$ and $\nu$ are $\nt$-parallel, we obtain  for $i\in\{2,3\}$: $\nabla^g_{\tilde\xi_i}\tilde\xi_1=0$ and $\nabla^g_{\tilde\xi_i}\nu=-(\tau_{\tilde\xi_i})_*\nu=0$. Therefore
$\tilde \xi_1$ and $\nu$ define on $S$ a unit  Killing vector field and a $2$-form which will be denoted by $\xi$ and $\nu^S$. They are both parallel with respect to the connection $\nabla^{\tau^S}=\nabla^{g^S}+\xi\wedge\nu^S$ on $S$. 
For every tangent vector on $M$ we have
$$
\nabla^g_X \tilde\xi_1 = - \tau_X \tilde\xi_1 = \tau_{\tilde\xi_1}X  =\nu(X) \ ,
$$
and similarly
$$\nabla^g_X\nu  =-(\tau_X)_*\nu=\nu_*\tau_X=\nu_*((g(X,\tilde\xi_1)\nu-\tilde\xi_1\wedge\nu(X))=
-\tilde\xi_1\wedge\nu(\nu((X))\ .
$$
On the other hand $\nu\in\Omega^+(\H)$ is a self-dual form. Denoting by $\alpha:=|\nu|_{\Lambda^2\H}$ (which is non-zero since $\nu$ is non-zero) and by $\Phi:=\frac1\alpha\nu$, we then have $\nu\circ\nu=-\alpha^2\Id_\H$ and $\Phi\circ\Phi=-\Id_\H$. The two relations above now read on $S$:
\be\label{als}
\nabla^{g^S}_X \xi =\alpha\Phi(X),\qquad \nabla^{g^S}_X\Phi=\alpha\xi\wedge X\ .
\ee
Hence, $(\xi, \Phi)$ defines an $\alpha$-Sasaki structure on $(S,g^S)$ (cf. Definition \ref{a-sas}). By Lemma \ref{par} and Lemma \ref{Lie}, there is a local Riemannian fibration $\pi:(S,g^S)\to(K,g^K)$ with fibers tangent to $\xi$, and $\Phi$ descends to a Kähler structure $\omega^K$ on $(K,g^K)$.

Using the orthonormal basis $\{\tilde\xi_1,\tilde\xi_2,\tilde\xi_3\}$ of $\V$, we can write $\f=\tilde\xi_1\wedge\tilde\xi_2\wedge\tilde\xi_3+\sum_{i=1}^3\tilde\xi_i\wedge\tilde\beta_i$. Clearly $\tilde\beta_i$ are $\nt$-parallel, so in particular they define $\nabla^{\tau^S}$-parallel $2$-forms on $S$ spanning $\Lambda^+\H$ at every point. That shows that $R^{\tau^S}$ seen as symmetric endomorphism of $\Lambda^2\T S=\Lambda^2\H\oplus \xi\wedge \H$, vanishes on the last summand (as $\xi$ is parallel), and its image commutes with $\Lambda^+\H$, i.e. is contained in $\Lambda^-\H$. Consequently, $R^{\tau^S}(\beta)=0$ for every $\beta\in\Omega^+(\H)$. 

On the other hand, the general curvature relation \eqref{rpb} applied to the Riemannian submersion $(S,g^S)\to(K,g^K)$ reads:
\be\label{rr}R^{\tau^S}(\pi^*\beta)=\pi^*(R^{g^K}(\beta))+4\alpha^2 \la\beta,\omega^K\ra_{\Lambda^2\T K}\pi^*\omega^K\ .\ee
for every $\beta\in\Omega^2(K)$. We thus obtain that 
\be\label{rg}0=R^{g^K}(\beta)+4\alpha^2\la\beta,\omega^K\ra_{\Lambda^2\T K}\omega^K\qquad\forall\beta\in\Omega^+(K)\ .\ee
Taking the scalar product with $\omega^K$ in \eqref{rg} gives $R^{g^K}(\omega^K)=-8\alpha^2\omega^K$, thus showing as before that $(K,g^K)$ is Kähler-Einstein with positive scalar curvature $32\alpha^2$. Note that this is consistent with the computation in \cite[Thm. 7.4]{MS24}. 

Conversely, let $(K,g^K,\omega^K)$ be a 4-dimensional Kähler-Einstein manifold with positive scalar curvature $\scal^K$. We denote by $\a:=\sqrt{\frac{\scal^K}{32}}$ and let $\zeta$ be a primitive of $2\alpha\omega^K$ on some open set $K_0$. Consider the Riemannian metric on $S:=\R\times K_0$ given by 
$$g^S:=(dt+\pi^*\zeta)^2+\pi^*g^K\ ,$$
where $\pi$ is the projection of the second factor, and denote by $\xi$ the metric dual of $dt+\pi^*\zeta$ and by $\Phi$ the skew-symmetric endomorphism of $\T S$ corresponding to $\pi^*\omega^K$. Then
$(S, g^S, \xi, \Phi)$ is a $5$-dimensional $\alpha$-Sasaki manifold. Indeed, $\xi$ is Killing and satisfies $d\xi=d\pi^*\zeta=2\alpha\pi^*\omega^K$, thus showing that $\xi$ satisfies the first equation in \eqref{als}, i.e. it is parallel with respect to the connection $\nabla^{g^S}+\tau^S$, where $\tau^S:=\alpha\xi\wedge \Phi$. We also have $(\tau^S_\xi)_*\Phi=\alpha\Phi_*\Phi=0$ so by Lemma \ref{par}, $\Phi$ is $\nabla^{\tau^S}$-parallel, which immediately gives the second equation in \eqref{als}. 

The curvature operator $R^K$ maps $\Lambda^+\T K$ to itself since $(K,g^K,\omega^K)$ is Kähler-Einstein, vanishes on $\Lambda^{(2,0)+(0,2)}\T K$, and maps $\omega^K$ to $\frac14\scal^K\omega^K=-8\alpha^2\omega^K$. By \eqref{rr} we then obtain that $R^{\tau_S}$ vanishes on $\Lambda^+\H$, where $\H$ denotes the horizontal distribution $\H:=\widetilde{\T K}$. Using the pair symmetry of $R^{\tau^S}$ we thus obtain that the restriction of $\nabla^{\tau^S}$ to $\Lambda^+\H$ is flat, so one can find (locally defined) $\nabla^{\tau^S}$-parallel $2$-forms $\beta_1,\beta_2,\beta_3\in \Omega^+(\H)$ satisfying the anti-quaternionic relations $[\beta_i, \beta_j] = - 2\beta_k$
for every even permutation $(i,j,k)$ of $\{1,2,3\}$.

Define $(M,g): = \R^2 \times (S,g^S) $. Every tensor on $S$ extends in a canonical way to a tensor on $M$ constant along $\R^2$. Let $\xi_1$ be the extension to $M$ of the Sasaki vector field $\xi$, and let $\{\xi_2,\xi_3\}$ be an orthonormal
$\nabla^g$-parallel frame on the $\R^2$ factor. 
Define the $3$-form 
\be\label{tau5}\tau: = \alpha\xi_1 \wedge \Phi\in\Omega^3(M)\ .\ee 
Then $\nabla^\tau: = \nabla^{g} + \tau$ extends to $M$ the canonical connection of the $\alpha$-Sasaki structure on $S$. It follows that $\xi_1, \xi_2, \xi_3$ are $\nabla^\tau$-parallel vector fields and that $\Phi$ is a $\nabla^\tau$-parallel $2$-form. 

Then the 3-form
$\varphi $ defined by
\be\label{phi5}\varphi:=\xi_1 \wedge \xi_2 \wedge \xi_3 + \sum^3_{i=1} \xi_i \wedge \beta_i \ee
defines a $\G_2$-structure on $(M,g)$ by Lemma \ref{phicom}, which is $\nabla^\tau$-parallel since $\xi_i$ and $\beta_i$ are all $\nt$-parallel by construction.

\subsection{The case where $\gamma_1,\gamma_2,\gamma_3$ are linearly independent} In this last case, the constant $a$ in Lemma \ref{com} is non-zero, so the $2$-forms $\Phi^H_i := -\frac{2}{a} \gamma_i \in \Omega^+(\H)$ satisfy the 
commutator relations $[\Phi^H_i, \Phi^H_j ] = - 2 \Phi^H_k$ for every even permutation $(i,j,k)$ of $\{1,2,3\}$. By Lemma \ref{phicom} we obtain $(\Phi^H_i)^2=-\Id_\H$ for every $i\in\{1,2,3\}$ and $\Phi^H_i \circ \Phi^H_j =-\Phi^H_j \circ \Phi^H_i= - \Phi^H_k$  for every even permutation $(i,j,k)$ of $\{1,2,3\}$. In addition we define $\Phi_i := \Phi_i^H - \xi_j \wedge \xi_k$. It follows $\xi_k = - \Phi_i(\xi_j) = \Phi_j(\xi_i)$ and 
$ \Phi_i^2 = - \Id_\H + \xi_i \otimes \xi_i$. An easy calculation gives
$$
\Phi_k (X)  = - \Phi_i \circ \Phi_j (X) + g (\xi_i, X) \xi_j \ .
$$
Computing $d\xi_i$ by means of \eqref{dxi} we
obtain
$$
d\xi_i = 2 \xi_i \lrcorner \tau = 2a \xi_j \wedge \xi_k + 2 \gamma_i
=  2a \xi_j \wedge \xi_k - a \Phi^H_i
=  a \xi_j \wedge \xi_k - a \Phi_i \ .
$$
By Definition \ref{3-sas} (cf. also  \cite{AD20} or \cite[\S 2.11]{MS24}), the tuple $(\xi_i, \Phi_i)$ defines a $3\text{-}(\alpha, \delta)$-Sasaki structure, where
$\alpha$ and $\delta$ are determined by $a=2(\alpha-\delta)$ and $-a=2\alpha$.
Thus $\alpha = -\frac{a}{2}$ and $\delta = -a$, i.e. we are in the special 
case where $\delta = 2\alpha\ne 0$, the so-called {\em parallel} $3$-$(\alpha,\delta)$-Sasaki
manifolds (cf. \cite[Def. 2.3.2]{AD20}).

Conversely assume that $(M^7,g,\xi_i, \Phi_i),\ i\in\{1,2,3\}$,  is a parallel $3\text{-}(\alpha, \delta)$-Sasaki manifold, i.e. $\delta=2\alpha \neq 0$. Set $\gamma_i :=\alpha \Phi^H$  and define a $3$-form $\tau$ via 
\be\label{tau4}\tau := 4\alpha \xi_1 \wedge \xi_2 \wedge \xi_3 + \frac12 \sum^3_{i=1} \xi_i \wedge d\xi_i\.\ee 
Then $\nabla^\tau := \nabla^g + \tau $ is a metric connection with parallel 
skew-symmetric torsion $\tau$ (cf. \cite[Cor. 4.4.2]{AD20}). It follows that
the vector fields $\xi_i$ and the $2$-forms $\Phi^H_i$ are $\nabla^\tau$-parallel.
From the definition of a $3\text{-}(\alpha, \delta)$-Sasaki manifold it is clear
that the forms $\Phi_i^H$ have square norm $2$ and satisfy the anti-quaternionic 
relations $\Phi^H_i \circ \Phi^H_j = - \Phi^H_k$ for every even permutation
$(i,j,k)$ of $\{1,2,3\}$.

Take any matrix $B = (b_{ij}) \in \SO(3)$ and define $2$-forms $\beta_i$
by $\beta_i := \sum^3_{i=1} b_{ij} \Phi^H_j$. It follows that the forms $\beta_i$ are
again $\nabla^\tau$-parallel and $[\beta_i,\beta_j] = -2\beta_k$ for all even
permutations $(i,j,k)$ of $\{1,2,3\}$. By Lemma \ref{phicom}, the 3-form $\varphi$ defined by
\begin{equation}\label{phi4}
\varphi := \xi_1 \wedge \xi_2 \wedge \xi_3 + \sum^3_{i=1} \xi_i \wedge \beta_i \ .
\end{equation}
is a $\G_2$-structure compatible with the metric $g$ of $M$.

Since the $2$-forms $\Phi^H_i$ and the vector fields $\xi_i$ are $\nabla^\tau$-parallel, the $\G_2$-form $\varphi$ is $\nabla^\tau$-parallel as well. One can check that $\f$ is co-closed if and only if the matrix $B\in\SO(3)$ is symmetric (i.e. $B=I_3$ or is the matrix of an orthogonal symmetry with respect to a line in $\R^3$).

\section{The case $\dim (\Par (\nabla^\tau))=0$}\label{xxx}

In this last section we will assume that there are no non-zero $\nabla^\tau$-parallel vector fields on $M$, i.e. 
the $\hol$-representation has no trivial summand.
Since $\dim M$ is odd the holonomy algebra $\hol$ cannot be abelian in this case.
According to Friedrich \cite[\S 2]{F07}, $\hol$ is then isomorphic to one of the algebras in the following list:
$$\mathfrak{g}_2,\quad
\su_c(2), \quad  \u(1) \oplus \su_c(2),\quad  \so_{ir}(3)  \quad
\mbox{or}  \quad
\su(2) \oplus  \su_c(2) \ .
$$

In the generic case $\hol = \mathfrak{g}_2$, since $\G_2$ has only one trivial 1-dimensional summand in $\Lambda^3\R^7$, the torsion is proportional to the $\G_2$-form, i.e.
$\tau = \lambda \varphi$ for some real constant $\lambda$. Since 
$\varphi$ is $\nabla^\tau$-parallel we obtain
$$
\nabla^g_X\varphi = -(\tau_X)_* \varphi = - \lambda (\varphi_X)_*\varphi =
-\lambda \sum (e_i \lrcorner X \lrcorner \varphi) \wedge (e_i \lrcorner \varphi)
= - 3\lambda * \varphi \ ,
$$
e.g. using (2.13)  in \cite{AS12}. It follows that the $\G_2$-structure
is either torsion free, for $\lambda=0$, or nearly parallel $\G_2$, for
$\lambda \neq 0$. Both classes of manifolds provide examples for connections with parallel skew-symmetric torsion with holonomy contained in $\G_2$.

If the holonomy algebra is isomorphic to $\su_c(2)$
or $\u(1) \oplus \su_c(2)$, it follows from \cite{F07} that $\nabla^\tau$ has parallel curvature, so the manifold has to be a  naturally reductive locally homogeneous space.

In the case $\hol =  \so_{ir}(3)$ the $\G_2$-structure is automatically
nearly parallel $\G_2$, since the space of $\SO_{ir}(3)$-invariant 
elements in $\Lambda^3\R^7$ is 1-dimensional. Moreover, the curvature
tensor $R^\tau$ turns out to be $\SO_{ir}(3)$-invariant, so we are
in the naturally reductive locally homogeneous case. It can be shown that the manifold is locally isometric
to the Berger space $\SO(5)/\SO_{ir}(3)$ (see \cite[Prop. 4.10]{MS24},
or \cite[Thm. 8.1]{F07} in the simply connected case).

It remains to study the case $\hol = \su(2) \oplus  \su_c(2)$. In this case the tangent bundle of $M$ decomposes in a $\nt$-parallel orthogonal direct sum $\T M=\V\oplus \H$ of oriented sub-bundles with $\dim(\V)=3$ (the orientation of $\V$ is the one defined by the restriction of $\f$). Let $\{\xi_1,\xi_2,\xi_3\}$ be a local orthonormal basis of $\V$. Then the 3-form $\vol^\V: = \xi_1 \wedge \xi_2 \wedge \xi_3$ does not depend on the choice of the basis and is $\nt$-parallel. By \eqref{fb} one can write 
\be\label{phixy}\varphi = \xi_1 \wedge \xi_2 \wedge 
\xi_3 + \sum \xi_i \wedge \beta_i\ ,\ee
where $\beta_i\in\Omega^+(\H)$ satisfy the relations
$\beta_i\circ \beta_j  = - \beta_k$ for every even permutation 
$(i, j, k) $ of   $\{1,2,3\}$. Note that the vector fields $\xi_i$ and the 2-forms $\beta_i$ are only locally defined and are not
$\nabla^\tau$-parallel in general.

From \cite[\S 2.6]{F07}, the space of $\hol$-invariant 3-forms is generated by $\vol^\V$ and $\f$. We can therefore write the torsion form  as 
\be\label{tauxy}\tau = x \xi_1 \wedge \xi_2 \wedge 
\xi_3 + y \sum \xi_i \wedge \beta_i\ee
for some real constants $x,y$.

The key point here will be to find a new metric connection $\nabla^{\tilde\tau}$ with skew-symmetric torsion, preserving $\V$, and which is flat on $\V$. In this way we will be able to choose a $\nabla^{\tilde \tau}$-parallel basis $\{\xi_1, \xi_2, \xi_3\}$ of $\V$ and then express the properties of the $\G_2$-form in terms of it.

Our Ansatz will be to take $\tilde \tau := \tau + \lambda \vol^\V$ for some $\lambda\in\R$. For any choice of $\lambda$, the corresponding connection $\nabla^{\tilde\tau}$ is metric, has skew-symmetric torsion, and preserves $\V$. 

\begin{epr}
    For $\lambda=-2(x+2y)$, the curvature $R^{\tilde \tau}$ of $\nabla^{\tilde \tau}$ vanishes on $\V$.
\end{epr}
\begin{proof}
For horizontal vectors $X, Y \in \H$ the Bianchi identity \cite[Eq. (3)]{CMS21} gives for every even permutation $(i,j,k)$ of $\{1,2,3\}$:
\begin{align*}
    R^\tau(X, Y, \xi_i, \xi_j)
    &=
    4 ( g(\tau_XY, \tau_{\xi_i}\xi_j)  + g(\tau_Y\xi_i, \tau_X \xi_j )
         + g(\tau_{\xi_i}X, \tau_Y \xi_j )) \\[.8ex]
         &=
         4\left( g(y \sum_{a=1}^3 \xi_a \beta_a(X,Y) ) , x \xi_k) + y^2 g(\beta_i(Y), \beta_j(X))
           - y^2 g(\beta_i(X), \beta_j(Y))\right)\\
         &=
         4(xy + 2y^2 ) \beta_k(X,Y)\ .
\end{align*}
Note that we cannot apply the Bianchi formula directly to $R^{\tilde \tau}$ since $\tilde\tau$ is not parallel with respect to $\nabla^{\tilde\tau}$. For computing the curvature $R^{\tilde \tau}$ we will first obtain formulas for the vertical part of commutators
of two horizontal vector fields $X,Y$:
$$
[X,Y]^\V  = (\nabla^g_XY - \nabla^g_YX)^\V =  -2 (\tau_XY)^\V = -2y \sum \xi_i \beta_i(X,Y)\ ,
$$
and similarly for a horizontal vector field $X$ and a vertical vector field $V$:
\be\label{xv}
[X,V]^\V = (\nabla^\tau_XV - \tau_XV - \nabla^\tau_VX + \tau_VX)^\V = \nabla^\tau_XV \ .
\ee
Note that $\nabla^{\tilde \tau}_X = \nabla^\tau_X$ holds for all horizontal
vectors $X$. Then we compute
\begin{align*}
g(R^{\tilde \tau}_{X,Y} \xi_i, \xi_j) &= 
   g(\nabla^{\tilde \tau}_X \nabla^{\tilde \tau}_Y \xi_i - \nabla^{\tilde \tau}_Y \nabla^{\tilde \tau}_X \xi_i - \nabla^{\tilde \tau}_{[X,Y]}\xi_i, \xi_j)\\
   &=
   g(R^\tau_{X,Y} \xi_i - \lambda \vol^\V_{[X,Y]} \xi_i, \xi_j)\\
   &=
   4(xy + 2y^2) \beta_k(X,Y) - \lambda g(\xi_k, [X,Y]) \\
   &=
   2y(2(x + 2y) + \lambda) \beta_k(X,Y)\ .
\end{align*}
We see that this expression vanishes for $\lambda = -2(x+2y)$. 
Next we compute using \eqref{xv}:
\begin{align*}
R^{\tilde \tau}_{X,V} \xi_i &=
   \nabla^\tau_X (\nabla^\tau_V \xi_i + \lambda \vol^\V_V \xi_i)
   -
   \nabla^\tau_V \nabla^\tau_X \xi_i - \lambda \vol^\V_V (\nabla^\tau_X\xi_i)
   -
   \nabla^\tau_{[X,V]} \xi_i - \lambda \vol^\V_{[X,V]}\xi_i \\
   &=
   R^\tau_{X,V}\xi_i + \lambda \vol^\V_{\nabla^\tau_XV}\xi_i - 
   \lambda \vol^\V_{[X,V]}\xi_i \\
   &=0\ .
\end{align*}

It remains to calculate the curvature $R^{\tilde \tau}$ on vertical 
vectors. The distribution $\V \subset \T M$
is totally geodesic and  the leaves have constant sectional curvature $K \ge 0$. The constant $K$ was computed in \cite{F07}. The torsion form $T=2\tau$ is written in loc. cit.
$$
2\tau = a \varphi + b \vol^\V  = (a+b) \vol^\V
+ a\sum \xi_i \wedge \beta_i
=2x \vol^\V + 2y \sum \xi_i \wedge \beta_i\ .
$$
By \eqref{tauxy}, the relation between $a,b$ and $x,y$ is $a+b=2x$ and $a=2y$. Then by \cite[Prop. 10.1]{F07} we have
$$
K=\frac14 (5a+b)^2 = (x + 4y)^2 \ .
$$
We can restrict the curvature calculation to the $3$-dimensional leaves of $\V$. The connection
$\nabla^{\tilde \tau}$ on $\V$ can be written as
$
\left. \nabla^{\tilde \tau} \right|_\V = \nabla^g_\V + (x+\lambda) \vol^\V
$.
A standard calculation then gives
$$
R^{\tilde \tau}_{V_1, V_2} = R^g_{V_1, V_2} + (x + \lambda)^2 V_1 \wedge V_2
= ((x + \lambda)^2 - K ) V_1 \wedge V_2
$$
If we take $\lambda = -2(x+2y)$ we get $x+\lambda = - x - 4y$, i.e.
$(x+\lambda)^2 = k$. We see that for this choice of $\lambda$ the curvature
of $\nabla^{\tilde \tau} $ on the bundle $\V$ vanishes.
\end{proof}

Since the bundle $\V$ is flat with respect to the connection $\nabla^{\tilde \tau}$,
we can choose a local orthonormal frame $\{\xi_1, \xi_2, \xi_3\}$ of  
$\nabla^{\tilde \tau}$-parallel sections of $\V$.  By Lemma \ref{xik}, the
vector fields $\xi_i$ are all Killing and their differentials can be computed as follows:
$$
d\xi_i =  2\tilde \tau_{\xi_i} = 2(x+\lambda) \xi_j \wedge \xi_k  + 2y \beta_i
= -2(x+4y)  \xi_j \wedge \xi_k + 2y \beta_i \ .
$$
From the last equation we see that $(\xi_i, \beta_i)$ defines a
$3$-$(\alpha, \delta)$-Sasaki structure as in Definition \ref{3-sas}, with $\alpha: = y$,  
$\delta: = x + 4y$ and $\Phi^H_i:=\beta_i$.

Note that if $\delta = 2\alpha$, then $\lambda = -2(x+2y)=0$. Indeed 
for these particular $3$-$(\alpha, \delta)$-Sasaki structures, the so-called
parallel $3$-$(\alpha, \delta)$-Sasaki structures, the vector fields
$\xi_i$  are $\nabla^\tau$ parallel and it is not necessary
to modify the torsion form $\tau$, i.e. we can take $\lambda = 0$.

Conversely, we have the following statement (see \cite[Rem. 4.4.3]{AD20}
and \cite[Thm.4.5.1]{AD20}).

 \begin{epr}\label{ad-thm}
  For any  $3$-$(\alpha, \delta)$-Sasaki structure $(\xi_i, \Phi^H_i)$ on a 
  $7$-dimensional Riemannian manifold $(M,g)$ there exists a canonical connection $\nabla^\tau$ with parallel skew-symmetric torsion
  $\tau$ defined as
 \be\label{tor-ad}
\tau :=  (\delta - 4\alpha) \vol^\V  + \alpha \sum \xi_i \wedge \Phi^H_i \ .
 \ee   
 Moreover, there is  an associated cocalibrated and $\nabla^\tau$-parallel $\G_2$-structure $ \varphi$ defined by
 $$
 \varphi := \vol^\V + \sum \xi_i \wedge \Phi^H_i \ .
 $$
 \end{epr}

The torsion form $\tau$ in Proposition \ref{ad-thm} coincides with the one defined in \eqref{tauxy} for $\beta_i=\Phi^H_i$, $y=\alpha$, and $x=\delta-4\alpha$. The $\G_2$-form
$\varphi$ is the $\G_2$-form defined in \eqref{phixy} for $\Phi^H_i=\beta_i$.

\labelsep .5cm


\begin{thebibliography}{22}

\bibliographystyle{alpha}

\bibitem{AD20}
I. Agricola, G. Dileo. 
{\sl Generalizations of 3-Sasakian manifolds and skew torsion}.
Adv. Geom. {\bf 20}, no. 3, 331--374 (2020). 

\bibitem{AS12}
B. Alexandrov, U. Semmelmann.
{\sl Deformations of nearly parallel $\G_2$-structures}.
Asian J. Math. {\bf 16}, no. 4, 713--744 (2012). 

\bibitem{BM01}
F. Belgun, A. Moroianu.
{\sl Nearly Kähler manifolds with reduced holonomy}.
Ann. Global Anal. Geom. {\bf 19}, 307--319 (2001).

\bibitem{B87} R. Bryant.
{\sl Metrics with exceptional holonomy}. 
Ann. of Math. (2) {\bf 126}, 525--576 (1987).

\bibitem{B06} R. Bryant.
{\sl Some remarks on $\G_2$-structures}. 
Proceedings of Gökova Geometry-Topology Conference 2005, 75–-109, Gökova Geometry/Topology Conference (GGT), Gökova, 2006.

\bibitem{CGT22}
A. Clarke, M. Garcia-Fernandez, C. Tipler.
{\sl T-dual solutions and infinitesimal moduli of the $\G2$-Strominger system}.
Adv. Theor. Math. Phys. {\bf 26}, no. 6, 1669--1704  (2022).

\bibitem{CMS21}
R. Cleyton, A. Moroianu, U. Semmelmann. 
{\sl Metric connections with parallel skew-symmetric torsion}.
Adv. Math. {\bf 378}, Paper No. 107519, 50 pp (2021). 

\bibitem{CS04}
R. Cleyton, A. Swann.
{\sl Einstein metrics via intrinsic or parallel torsion}.
Math. Z. {\bf 247}, no. 3, 513--528 (2004).

\bibitem{FG82} M. Fernández, A. Gray. 
{\sl Riemannian manifolds with structure group $\G_2$}. 
Annali di Matematica pura ed applicata {\bf 132}, 19--45 (1982).

\bibitem{FF25}
A. Fino, U. Fowdar.
{\sl Some remarks on strong $\G_2$-structures with torsion}.
arXiv:2502.06066 (2025).


\bibitem{FKMS97}
Th. Friedrich, I. Kath, A. Moroianu, U. Semmelmann.
{\sl On nearly parallel $\G_2$-structures}.
J. Geom. Phys. {\bf 23}, no. 3-4, 259--286 (1997).

\bibitem{FI02}
Th. Friedrich, S. Ivanov.
{\sl Parallel spinors and connections with skew-symmetric torsion in string theory}.
Asian J. Math. {\bf 6}, no. 2, 303--335 (2002).


\bibitem{F07}
Th. Friedrich.
{\sl $\G_2$-manifolds with parallel characteristic torsion}.
Differential Geom. Appl. {\bf 25}, no. 6, 632--648 (2007). 

\bibitem{IPU25}
S. Ivanov, A. Petkov, L. Ugarte.
{\sl Parallel torsion and $\G_2, \mathrm{Spin}(7)$ instantons}.
arXiv:2509.10623  (2025)


\bibitem{MS24}
A. Moroianu, P. Schwahn.
{\sl Submersion constructions for geometries with parallel skew torsion}. 
arxiv:2409.14421 (2024).


\bibitem{OLMS20}
X. de la Ossa, M. Larfors, M.  Magill, E. Svanes.
{\sl Superpotential of three dimensional $N=1$ heterotic supergravity}.
J. High Energy Phys. {\bf 195}, no. 1, 26 pp. (2020).

\bibitem{PR23}
F. Podestà, A. Raffero.
{\sl Bismut Ricci flat manifolds with symmetries}.
Proc. Roy. Soc. Edinburgh Sect. A {\bf 153}, no. 4, 1371--1390  (2023).


\bibitem{R93} R.~Reyes-Carrión.
\emph{Some special geometries defined by Lie groups}. 
PhD Thesis, Oxford, 1993.

\bibitem{S20}
R. Storm.
{\sl The classification of $7$- and $8$-dimensional naturally reductive spaces}.
Canad. J. Math. {\bf 72}, no. 5, 1246--1274  (2020). 


\bibitem{S86} 
A. Strominger. 
{\sl Superstrings with torsion}.
Nucl. Phys. B {\bf 274}, 253--284 (1986).



\end{thebibliography}
\end{document}